\documentclass[11pt,reqno]{amsart}

\allowdisplaybreaks

\usepackage{amssymb,amsmath,amsthm}
\usepackage[margin=1.1in]{geometry}
\usepackage{nicefrac}
\usepackage{bm}
\usepackage{mathtools}
\usepackage{comment}
\usepackage[utf8]{inputenc}
\usepackage[parfill]{parskip}
\usepackage{enumitem}
\usepackage{xcolor}
\usepackage[backref]{hyperref}
\usepackage[alphabetic,nobysame]{amsrefs}

\newcommand*{\eqcomment}[2][\qquad]{#1 &&\left(\text{#2}\right)}

\newcommand*{\be}{\begin{equation}}
\newcommand*{\ee}{\end{equation}}
\newcommand*{\de}[1]{\,\mathrm{d}#1}

\theoremstyle{definition}
\newtheorem{example}{Example}[section]
\newtheorem{definition}[example]{Definition}
\newtheorem{remark}[example]{Remark}

\theoremstyle{plain}
\newtheorem{theorem}[example]{Theorem}
\newtheorem{proposition}[example]{Proposition}
\newtheorem{notation}[example]{Notation}
\newtheorem{lemma}[example]{Lemma}
\newtheorem{corollary}[example]{Corollary}

\numberwithin{equation}{section}

\newcommand*{\R}{\mathbb{R}}

\newcommand*{\Z}{\mathbb{Z}}

\newcommand*{\N}{\mathbb{N}}
\newcommand*{\Q}{\mathbb{Q}}

\newcommand*{\eps}{\varepsilon}
\newcommand*{\PP}{\mathcal{P}}

\newcommand*{\ignore}[1]{}

\newcommand{\x}{\mathrm{x}}

\newcommand{\vecx}{\mathrm{x}}
\newcommand{\suchthat}{\bm{:}}

\newcommand{\smooth}[1]{#1^{\text{smooth}}}

\newcommand{\Pf}{\mathrm{Pfaff}}
\newcommand{\conncomp}{\mathcal{C}}

\newcommand{\Pfaff}{\mathcal{X}}

\title{Partitioning Theorems for Sets of Semi-Pfaffian Sets, with Applications}
\author{Martin Lotz}
\address{(Lotz) Warwick Mathematics Institute, University of Warwick, UK}
\email{martin.lotz@warwick.ac.uk}
\author{Abhiram Natarajan}
\address{(Natarajan) Warwick Mathematics Institute, University of Warwick, UK}
\email{abhiram.natarajan@warwick.ac.uk}
\author{Nicolai Vorobjov}
\address{(Vorobjov) University of Bath, UK and St. Petersburg Department of V. A. Steklov Institute of Mathematics of the Russian Academy of Sciences (PDMI RAS), Russia}
\email{masnnv@bath.ac.uk}

\date{}

\begin{document}
\begin{abstract}
We generalize the seminal polynomial partitioning theorems of Guth and Katz \cites{guth2015erdHos, guth2015polynomial} to a set of semi-Pfaffian sets. Specifically, given a set $\Gamma \subseteq \R^n$ of $k$-dimensional semi-Pfaffian sets, where each $\gamma \in \Gamma$ is defined by a fixed number of Pfaffian functions, and each Pfaffian function is in turn defined with respect to a Pfaffian chain $\vec{q}$ of length $r$, for any $D \ge 1$, we prove the existence of a polynomial $P \in \R[X_1, \ldots, X_n]$ of degree at most $D$ such that each connected component of $\R^n \setminus Z(P)$ intersects at most $\sim \frac{|\Gamma|}{D^{n - k - r}}$ elements of $\Gamma$. Also, under some mild conditions on $\vec{q}$, for any $D \ge 1$, we prove the existence of a Pfaffian function $P'$ of degree at most $D$ defined with respect to $\vec{q}$, such that each connected component of $\R^n \setminus Z(P')$ intersects at most $\sim \frac{|\Gamma|}{D^{n-k}}$ elements of $\Gamma$. To do so, given a $k$-dimensional semi-Pfaffian set $\Pfaff \subseteq \R^n$, and a polynomial $P \in \R[X_1, \ldots, X_n]$ of degree at most $D$, we establish a uniform bound on the number of connected components of $\R^n \setminus Z(P)$ that $\Pfaff$ intersects; that is, we prove that the number of connected components of $(\R^n \setminus Z(P)) \cap \Pfaff$ is at most $\sim D^{k+r}$. Finally, as applications, we derive Pfaffian versions of Szemer\'edi-Trotter-type theorems, and also prove bounds on the number of joints between Pfaffian curves.
\end{abstract}

\maketitle

\tableofcontents 

\section{Introduction}

Incidence combinatorics has recently experienced a surge of activity, largely fuelled by the introduction of powerful algebraic techniques (see, for example, \cites{guth2016polynomial,sheffer2022polynomial}). While traditionally restricted to simple geometric objects such as points and lines, the focus in incidence geometry has shifted towards questions involving higher-dimensional algebraic or semi-algebraic sets. Semi-algebraic sets possess convenient \emph{tame topological} properties, such as triangulability and stratifiability, which play an important role in incidence questions. As suggested in \cite{grothendieck1997esquisse}, if one were to look for other collections of subsets of $\R^n$ that also possess the tame topological properties of semi-algebraic sets, one is naturally led to \emph{o-minimal geometry}, an axiomatic generalization of semi-algebraic geometry. 

Having its genesis in model theory, the notion of o-minimality isolates key axioms such that any collection of subsets of $\R^n$ satisfying these axioms -- such a collection is called an \emph{o-minimal structure}, and elements of an o-minimal structure are called \emph{definable sets} -- share the tame topological properties that semi-algebraic sets possess. We refer the reader to \cite{markermodeltheory} and \cite{van1998tame} for background in model theory and o-minimal geometry. There is a growing body of work investigating incidence problems where the involved sets are definable over arbitrary o-minimal structures or distal structures over $\R$ \cites{basu2018minimal,chernikov2020cutting,chernikovregularity2020,chernikov2021ramsey,basit2021Zar,anderson2023combinatorial,balsera2023incidences,chernikov2024modelduke};  see \cite[\S 6]{scanlon} for an overview. However, compared to the algebraic and semi-algebraic settings, progress in o-minimal incidence combinatorics has been slow. 

A particularly useful algebraic tool in incidence geometry is the celebrated \emph{polynomial partitioning theorem}, first introduced in ground-breaking work by Guth and Katz \cite{guth2015erdHos} to resolve the distinct distances problem and the joints problem. This method was later generalized in \cite{guth2015polynomial} and \cite{blagojevic2017polynomial}. In the statement of this theorem, and in what follows, we denote by $Z(P_1, \dots, P_s)$ the set of common zeros in $\R^n$ of functions $P_1,\dots,P_s$.

\begin{theorem}[Polynomial partitioning \cites{guth2015erdHos,guth2015polynomial}]
\label{thm:guth-polynomial-partitioning}
Let $\Gamma$ be a finite set of $k$-dimensional real algebraic sets in $\R^n$, where each $\Pfaff \in \Gamma$ is defined by at most $m$ polynomials, each of degree at most $\beta$. For any integer $D \ge 1$, there exists a non-zero polynomial $P \in \R[X_1, \ldots, X_n]$ of degree at most $D$ and a constant $C = C(n, m, \beta)$, such that each connected component of $\R^n \setminus Z(P)$ intersects at most $CD^{k - n}|\Gamma|$ elements of $\Gamma$.
\end{theorem}

Theorem \ref{thm:guth-polynomial-partitioning} can be interpreted as follows. It is well known (see Theorem \ref{thm:optm}) that given a polynomial $P \in \R[X_1, \ldots, X_n]$ of degree at most $D$, $\R^n \setminus Z(P)$ can have at most $\mathcal{O}_n(D^n)$ connected components. It is also known (see Theorem \ref{thm:barone-basu-ST}) that out of these $\mathcal{O}_n(D^n)$ connected components, any $k$-dimensional real algebraic set $\Pfaff \subseteq \R^n$ can intersect at most $\mathcal{O}_{n,\Pfaff}(D^k)$ of them.\footnote{The subscript $\Pfaff$ in the big-O notation is shorthand to denote that we are considering various parameters that determine the \emph{complexity} of $\Pfaff$ to be fixed; our theorem statements are more precise.} The polynomial partitioning theorem says that given \emph{any} finite set $\Gamma$ of $k$-dimensional real algebraic sets, and given any $D \ge 1$, \emph{there exists} a polynomial $P$ of degree at most $D$ such that the intersections of algebraic sets with connected components are equidistributed. More specifically, since each $\Pfaff \in \Gamma$ can intersect at most $\mathcal{O}_{n, \Pfaff}(D^k)$ connected components of $\R^n \setminus Z(P)$, the total number of intersections between elements of $\Gamma$ and connected components of $\R^n \setminus Z(P)$ is at most $\mathcal{O}_{n,\Pfaff}(|\Gamma|D^k)$; Theorem \ref{thm:guth-polynomial-partitioning} promises that there exists a polynomial $P$ that ensures that these intersections are equidistributed among the $\mathcal{O}_n(D^n)$ connected components of $\R^n \setminus Z(P)$, i.e., there are at most $\mathcal{O}_{n,\Pfaff}(|\Gamma|D^{k-n})$ elements of $\Gamma$ intersecting each connected component of $\R^n \setminus Z(P)$. While not stated this way, Theorem \ref{thm:guth-polynomial-partitioning} holds even if $\Gamma$ is a finite set of semi-algebraic sets, since any semi-algebraic set is contained in a real algebraic set of the same dimension.

The polynomial partitioning theorem gives us a divide and conquer technique; it allows to break a problem instance into smaller sub-instances, solve each sub-instance, and then put together the local solutions to get the global answer. 
The polynomial partitioning technique has been a panacea for a large number of problems in discrete geometry \cites{guth2015erdHos,kaplan2012simple,sheffer2021distinct,tidor2022joints}, computational geometry/data structures \cites{afshani2023lower,agarwal2021efficient,doi:10.1137/19M1257548}, harmonic analysis \cites{guth2016restriction,guth2018restriction}, and others.

A long-standing open problem that has stymied progress on incidence questions involving sets definable over arbitrary o-minimal expansions of $\R$ is the unavailability of a polynomial partitioning type result for definable sets; see comments in \cites{basu2018minimal,chernikov2020cutting,anderson2023combinatorial,balsera2023incidences,chernikov2024modelduke}. Needless to say, a polynomial partitioning theorem that works in the o-minimal setting would enable tremendous progress in o-minimal incidence combinatorics and would also provide greatly simplified proofs of existing results. 

\begin{remark}[Difficulty in obtaining an o-minimal polynomial partitioning theorem] \label{rem:o-minimal-poly-part-difficult}
To prove a polynomial partitioning theorem in the o-minimal setting, a key step is to obtain a result of the type of Theorem \ref{thm:barone-basu-ST} in the case where $\Pfaff \subseteq \R^n$ is a $k$-dimensional definable set in any o-minimal structure. Specifically, given a polynomial $P \in \R[X_1, \ldots, X_n]$ of degree at most $D$, we would like an upper bound on the number of connected components of $\R^n \setminus Z(P)$ that $\Pfaff$ intersects. However, a result of \cite{basu2019zeroes} suggests that a uniform bound might not always be possible for definable sets in any arbitrary o-minimal structure. Specifically, \cite{basu2019zeroes} show that for every $0 \le k \le n-2$, and every sequence of natural numbers $(a_d)_{d \in \N}$, there is a regular compact and definable hypersurface $\Pfaff \subseteq \R \mathbb{P}^n$, a subsequence $(a_{d_m})_{m \in \N}$, and a sequence of homogeneous polynomials $(P_m)_{m \in \N}$ of degree $\mathrm{deg}(P_m) = d_m$, such that \[b_k(\Pfaff \cap Z(P_m)) \ge a_{d_m},\] where $b_k$ is the $k^{\text{th}}$ singular Betti number. This is a generalization of a result in \cite{gwozdziewicz1999number}. Moreover, in their construction, which constructs a semianalytic $\Pfaff$, $Z(P_m)$ intersects $\Pfaff$ transversely, thus making the topology of the intersection robust to small perturbations. See \cite[Section 3.4]{abhiramthesis} for a discussion of this issue.
\end{remark}

\subsection{Results: Partitioning Theorems}
In this article, we make progress on the problem of obtaining an o-minimal polynomial partitioning theorem by proving partitioning theorems for the specific case of sets defined by Pfaffian functions (see Section \ref{sec:pfaffian} for precise definitions). Introduced by Khovanski{\u\i} \cites{khovanskii1980class, khovanskiiicm} for the study of \emph{fewnomials} \cite{khovanskiui1991fewnomials} and the second part of \emph{Hilbert's sixteenth problem} \cite{ilyashenko}, Pfaffian functions satisfy triangular systems of first-order partial differential equations with polynomial coefficients. Pfaffian functions encompass a wide class of functions that contains, in particular, all Liouvillian functions (provided any sine or cosine functions involved are appropriately restricted). Pfaffian functions arise in many different mathematical areas. In model theory, they played an important role in Wilkie's Theorem~\cite{wilkie1996model}.

It is known that the \emph{Pfaffian structure}, i.e., the smallest collection of sets containing all semi-Pfaffian sets and that is stable under all structure operations, is o-minimal. In fact, it is also known that Pfaffian functions over any o-minimal expansion of the real field form an o-minimial structure \cite{speissegger1999pfaffian}. See \cite{karpinski1997polynomial} for an example in machine learning.

We generalize Theorem \ref{thm:guth-polynomial-partitioning} by proving Theorem \ref{thm:poly-part-bbz} which is a polynomial partitioning theorem for a set of semi-Pfaffian sets. Under a mild restriction on the Pfaffian chain, i.e. by restricting ourselves to \emph{algebraically independent} Pfaffian chains (see Definition \ref{defn:algebraically-independent-pfaffian-chain}), we also prove Theorem \ref{thm:pfaff-part-bbz} which is a partitioning theorem where, instead of using a polynomial to partition the given set of semi-Pfaffian sets, we use another Pfaffian function to partition; henceforth, we shall refer to this as \emph{Pfaffian partitioning}. Below, in Theorem \ref{thm:poly-part-guth}, we only state instantiations of Theorem \ref{thm:poly-part-bbz} and Theorem \ref{thm:pfaff-part-bbz} because it becomes easier to contrast with Theorem \ref{thm:guth-polynomial-partitioning}.

\begin{notation}[Pfaffian functions]
    We will denote by $\Pf_n(\alpha,\beta,r)$ the set of Pfaffian functions in $n$ variables with chain-degree $\alpha$, degree $\beta$, and order at most $r$. Also, if $\vec{q}$ is a Pfaffian chain of order $r$ and chain-degree $\alpha$, then $\Pf_{n}(\beta, \vec{q}) \subseteq \Pf_n(\alpha, \beta, r)$ will denote the set of all Pfaffian functions in $\Pf_n(\alpha,\beta,r)$ defined with respect to (w.r.t.) $\vec{q}$. See Section~\ref{sec:pfaffian} for the definition of Pfaffian functions and related concepts, including the notion of an algebraically independent Pfaffian chain.
\end{notation}

\begin{theorem}[Polynomial and Pfaffian partitioning of semi-Pfaffian sets - instantiations of Theorem \ref{thm:poly-part-bbz} and Theorem \ref{thm:pfaff-part-bbz}]
\label{thm:poly-part-guth}
Let $\Gamma$ be a collection of semi-Pfaffian sets in $\R^n$ of dimension $k$, where each $\Pfaff \in \Gamma$ is defined by at most $m$ Pfaffian functions from $\Pf_n(\alpha,\beta,r)$. 
\begin{enumerate}[label=(\Roman*),ref=(\Roman*)]
\item \label{point:normal-poly-part} For any $D \ge 1$, there is a non-zero polynomial $P \in \R[X_1, \ldots, X_n]$ of degree at most $D$, and a constant $C = C(n, m, \alpha, \beta, r)$, such that each connected component of $\R^n \setminus Z(P)$ intersects at most $C\beta_1$ elements of $\Gamma$, where
\[
\beta_1 = \begin{cases}
    \frac{|\Gamma|}{D^{n - k - r}} & \text{if $k > 0$} \\
    \frac{|\Gamma|}{D^n} & \text{if $k = 0$}
\end{cases}.
\]
\item \label{point:pfaffian-poly-part} Suppose that $\vec{q}$ is an algebraically independent Pfaffian chain (see Definition \ref{defn:algebraically-independent-pfaffian-chain}) of chain-degree $\alpha$ and order $r$, and suppose that all Pfaffian functions involved in defining the elements of $\Gamma$ are defined w.r.t. $\vec{q}$. For any $D \ge 1$, there is a non-zero Pfaffian function $P' \in \Pf_n(D, \vec{q})$, and a constant $C = C(n, m, \alpha, \beta, r)$, such that each connected component of $\R^n \setminus Z(P')$ intersects at most $C\beta_2$ elements of $\Gamma$, where
\[\beta_2 = \begin{cases}
    \frac{|\Gamma|}{D^{n-k}} & \text{if $k > 0$} \\
    \frac{|\Gamma|}{D^{n+r}} & \text{if $k = 0$}
\end{cases}.\]
\end{enumerate}
\end{theorem}

In our general theorems that are stated and proved in Section \ref{sec:main-proofs-subsection}, we have multiple collections of semi-Pfaffian sets, and we establish the existence of a partitioning polynomial (Theorem \ref{thm:poly-part-bbz}), and a partitioning Pfaffian function (Theorem \ref{thm:pfaff-part-bbz}), that simultaneously partitions all collections. Polynomial partitioning of multiple collections of real-algebraic sets was considered in \cite{blagojevic2017polynomial}.

Given a set $\Gamma \subseteq \R^n$, all partitioning theorems promise the existence of a function $P$ such that $Z(P)$ partitions $\R^n$ ensuring that each connected component of $\R^n \setminus Z(P)$ intersects a small fraction of the elements of $\Gamma$; in what follows, the exact value of this fraction will be called the \emph{quantitative guarantee} of the theorem.

\begin{remark}[Polynomial partitioning vs Pfaffian partitioning]
Theorem \ref{thm:poly-part-guth}-\ref{point:normal-poly-part}  generalizes Theorem \ref{thm:guth-polynomial-partitioning} since we recover Theorem \ref{thm:guth-polynomial-partitioning} by setting $r=0$ (c.f. Example \ref{example:pfaffian}-\ref{point:polynomial-example}). However, as the value of $r$ gets closer to $n-k$, the quantitative guarantee of Theorem \ref{thm:poly-part-guth}-\ref{point:normal-poly-part} gets closer to the trivial guarantee, i.e. $|\Gamma|$. On the other hand, it is noteworthy that, when $k \ge 1$, the quantitative guarantee in Pfaffian partitioning of a set of semi-Pfaffian sets, i.e. Theorem \ref{thm:poly-part-guth}-\ref{point:pfaffian-poly-part}, matches the quantitative guarantee of polynomial partitioning of a set of semi-algebraic sets in Theorem \ref{thm:guth-polynomial-partitioning}, \emph{irrespective of the value of $r$}; also when $k=0$, the quantitative guarantee is even better than what is obtained in Theorem \ref{thm:guth-polynomial-partitioning}. However, the caveat is that the object that accomplishes the partitioning of our set of semi-Pfaffian sets in Theorem \ref{thm:poly-part-guth}-\ref{point:pfaffian-poly-part}, namely the zero set of the Pfaffian function $P'$, is \emph{more topologically complex} than the real algebraic set that partitions the set of semi-algebraic sets in Theorem \ref{thm:guth-polynomial-partitioning}, which is the zero set of the polynomial $P$. Having said that, using some auxiliary tools developed in Section \ref{sec:preliminaries-tools} and due to the existence of Khovanski{\u\i}'s B\'ezout type bound for Pfaffian sets, by using \emph{constant-degree} Pfaffian functions to partition our set of semi-Pfaffian sets, we are able to control the complexity of the partitioning Pfaffian set. Pfaffian partitioning with constant degree Pfaffian functions is used in all applications shown in Section \ref{sec:applications}; see Section \ref{sec:constantsvshighdegree} for more details regarding the complexity of the partitioning set and how it influences our applications.
\end{remark}

As discussed in Remark \ref{rem:o-minimal-poly-part-difficult}, the key step for proving Theorem \ref{thm:poly-part-guth} is to prove a semi-Pfaffian analogue of Theorem \ref{thm:barone-basu-ST}. In the statement of Theorem \ref{thm:bezout-pfaffian} below, $b_0$ denotes the $0^{\text{th}}$ Betti number, i.e., the number of connected components.

\begin{theorem}[Semi-Pfaffian Barone-Basu type bound]
\label{thm:bezout-pfaffian}
Let $\vec{q}$ be a Pfaffian chain of chain-degree $\alpha$ and order $r$. Let $\Pfaff \subseteq \R^n$ be a $k$-dimensional semi-Pfaffian set defined by at most $m$ Pfaffian functions from $\Pf_n(\beta, \vec{q})$. Then for any $P \in \Pf_n(D, \vec{q})$, where $D \ge 1$, there exists a constant $C = C(n, m, \alpha, \beta, r)$ such that 
\[b_0\left(\Pfaff \cap (\R^n \setminus Z(P))\right) \le CD^{k+r}.\]
\end{theorem}

Setting $r=0$ in Theorem \ref{thm:bezout-pfaffian} shows that it generalizes Theorem \ref{thm:barone-basu-ST}. The bound in Theorem \ref{thm:bezout-pfaffian}, which obviously also works if $P \in \R[X_1, \ldots, X_n]$ of degree at most $D$, is obtained by setting up a system of Pfaffian equations such that the number of zeros of the system gives an upper bound on the number of connected components of $\Pfaff \cap (\R^n \setminus Z(P))$; the key to counting the number of solutions of this system is Khovanski{\u\i}'s B\'ezout type bound for Pfaffian sets, Theorem~\ref{thm:khovanskii}.

\begin{remark}[Regarding the `$r$' in the exponent of $D$ in bound of Theorem \ref{thm:bezout-pfaffian}]
    In Khovanski{\u\i}'s B\'ezout type bound for Pfaffian functions (Theorem \ref{thm:khovanskii}), the exponent in the bound depends on the order $r$ of the Pfaffian chain, i.e., the largest degree of the Pfaffian functions involved is raised to the exponent $r$. It is this dependence that appears in the bound in Theorem \ref{thm:bezout-pfaffian}. On the one hand, Khovanski{\u\i}'s bound is known to be at least exponential in $r$.
    On the other hand, while Theorem \ref{thm:barone-basu-ST} is indeed asymptotically tight, we do not prove that Theorem \ref{thm:bezout-pfaffian} is asymptotically tight.
\end{remark}

\begin{remark}[Asymptotics in theorem statements] \label{rem:asymptotics}
In our main theorems, there are a number of different parameters involved - chain degree $\alpha$, degree $\beta$, order $r$, number of equations defining our Pfaffian sets $m$, cardinality of the set of varieties/Pfaffian-sets $|\Gamma|$, etc. These paramaters are customarily separated into two categories -- \emph{algebraic} (e.g., degree $\beta$, order $r$), and \emph{combinatorial} (e.g., $|\Gamma|$). In general, one could be interested in various different asymptotic regimes for the parameters involved. However for several applications, especially in incidence combinatorics, discrete geometry and harmonic analysis, the situation of relevance is where $|\Gamma|$ goes to infinity, while the other parameters are fixed. In other words, there are a \emph{huge number} of objects, but the \emph{complexity} of the individual objects is small. This motivates the treatment of certain parameters as part of the constants in our statements.
\end{remark}

\subsection{Results: Applications of Partitioning Theorems} 
In Section \ref{sec:applications}, we derive two types of applications of our main theorems and other auxiliary tools proved in Section \ref{sec:preliminaries-tools}. In our first application, we study incidences between points and Pfaffian curves in $\R^n$. Incidence problems are a major topic in discrete geometry that are both interesting on their own right, and also have connections to other mathematical problems such as the sum-product problem \cite{elekessumproduct} and Kakeya conjectures \cite{wolffkakeya}; see \cite{dvirincidencesurvey} for a survey of the area. In Theorem \ref{thm:st-rn}, we prove a Szemer\'edi-Trotter type theorem where, with non-degeneracy assumptions, we give an upper bound on the number of incidences between points and Pfaffian curves in $\R^n$. In Corollary \ref{cor:st-plane-curves}, which is an instantiation of Theorem \ref{thm:st-rn}, we prove an upper bound on the number of incidences between points and Pfaffian curves in $\R^2$. Corollary \ref{cor:st-plane-curves} has milder non-degeneracy assumptions, and but for a small $\eps$ factor, generalizes \cite[Theorem 3.3]{sheffer2022polynomial} and the original Szemer\'edi-Trotter theorem \cite{szemereditrotter}, which is perhaps the first non-trivial incidence result. 

The setting of Corollary \ref{cor:st-plane-curves} is not new in literature as both \cite{basu2018minimal} and and \cite{chernikov2020cutting} prove Szemer\'edi-Trotter type theorems in arbitrary o-minimal/distal structures, albeit with different formulations and different non-degeneracy assumptions. We would like to emphasise that the techniques in \cites{basu2018minimal,chernikov2020cutting} are quite different and are arguably more involved; our partitioning theorem allows us to prove Pfaffian Szemer\'edi-Trotter type theorems by following the structure of the simpler methods that are used in the algebraic case. 

As our second application, in Theorem \ref{thm:pfaffian-joints}, we derive the first upper bound on the number of joints formed by $n$ sets of Pfaffian curves in $\R^n$. The joints problem (see Definition \ref{defn:joints} for a precise definition of joints) was originally posed in \cite{chazelle1992counting} and improved upon in multiple works, for example \cites{sharir1994joints, feldman2005improved}, with the final word on the number of joints for lines in $\R^3$ in \cite{guth2010algebraic}. The problem of counting joints is again both of intrinsic mathematical value, and also motivated by applications in computational geometry related to the \emph{hidden surface removal} problem; further, connections betwen the joints problem and the Kakeya problem have been exploited successfully \cites{wolffkakeya,bennettkakeya}.

Theorem \ref{thm:pfaffian-joints} is a generalization of the results of \cites{kaplan2010lines,quilodran2010joints} and \cite[Theorem 9]{yang2016generalizations}. As before, we are able to prove a generalization by following the structure of the proof in the algebraic case, albeit with appropriate modifications.

We discuss our applications and finer technical points regarding them in more detail in Section \ref{sec:applications-discussion}. Our applications, i.e., Theorem \ref{thm:st-rn}, Corollary \ref{cor:st-plane-curves}, and Theorem \ref{thm:pfaffian-joints}, demonstrate the efficacy of our partitioning theorems and other auxiliary results in producing Pfaffian versions of discrete geometry theorems. The generality of our main theorems suggests that it is reasonable to expect that there will be many more applications to combinatorial problems involving semi-Pfaffian sets.

\subsection{Acknowledgements} AN would like to specially thank Adam Sheffer for his comments and very helpful discussions. AN would also like to thank Saugata Basu and Joshua Zahl for comments on a preliminary draft, and Alexander Balsera for useful discussions. AN was supported by EPSRC Grant EP/V003542/1.

\section{Preliminaries}
\label{sec:preliminaries}

\begin{notation}[General notation] $\N$ denotes the set of natural numbers $\{0, 1, 2, \ldots \}$ and $\N_{>0}$ the positive integers. We use the notation $[n] := \{1, \ldots, n\}$. We denote by $\R[X_1, \ldots, X_n]_{(\le D)}$ the vector subspace of the polynomial ring $\R[X_1,\dots,X_n]$ containing polynomials of degree at most $D$. We use the notation $\vecx=(x_1,\dots,x_n)\in \R^n$ for a vector. 
\end{notation}

\subsection{Basic Definitions}
\label{sec:pfaffian}

Pfaffian functions are real or complex analytic functions that satisfy triangular systems of first-order partial differential equations with polynomial coefficients. 

\begin{definition}[Pfaffian functions]
\label{defn:pfaffian}
Let $\mathcal{U} \subseteq \R^n$ be an open set. A \emph{Pfaffian chain} of \emph{order} $r \in \N$ and \emph{chain-degree} $\alpha \in \N_{>0}$ over $\mathcal{U}$ is a sequence of functions $\vec{q} = (q_1, \ldots, q_r)$, $q_i \in C^{\infty}(\mathcal{U})$, such that there exist polynomials $P_{i, j} \in \R[X_1, \ldots, X_n, Y_1, \ldots, Y_j]_{(\le \alpha)}$ verifying
\begin{equation} \label{eqn:pfaffian-differential-condition}
\de q_j(\vecx) = \sum_{i=1}^n P_{i,j}(\x, q_1(\vecx), \ldots, q_j(\vecx))\,\de x_i,
\end{equation}
for all $j \in [r]$.

Let $\beta \in \N_{>0}$. A function $f(\vecx) = P(\vecx, q_1(\vecx), \ldots, q_r(\vecx))$, where $P \in \R[X_1, \ldots, X_n, Y_1, \ldots, Y_r]_{(\le \beta)}$, is called a \emph{Pfaffian function}.
We will say that $f$ is \emph{defined w.r.t. the chain $\vec{q}$}, has \emph{degree} $\beta$, has \emph{order} $r$, and has \emph{chain-degree} $\alpha$. We will also sometimes say that $f$ has \emph{degree} $(\alpha, \beta)$, and/or $f$ has \emph{format} $(\alpha, \beta, r)$.
\end{definition}

\begin{remark} There are a few different definitions of Pfaffian functions in the literature. Definition \ref{defn:pfaffian}, which matches the definition of Pfaffian functions in \cite{gabrielov2004complexity}, is slightly more restrictive than the original definition in \cite{khovanskiui1991fewnomials}. However, both definitions lead to nearly the same class of Pfaffian functions, albeit sometimes with different formats. In~\cites{MR1889560,MR1658452}, the authors allow for non-polynomial coefficients in Equation \eqref{eqn:pfaffian-differential-condition}.
\end{remark}

\begin{remark}[Domain $\mathcal{U}$ of Pfaffian functions \cite{gabrielov2004complexity}] The results in the paper are stated in terms of the format of the various semi-Pfaffian sets we consider. These results also depend on the domain $\mathcal{U}$. While Definition \ref{defn:pfaffian} imposes no restrictions on $\mathcal{U}$, thus allowing it to be arbitrarily complex causing the corresponding Pfaffian/semi-Pfaffian sets to also be arbitrarily complex, in the rest of the paper, we will always assume that the domain is `simple', like $\R^n$, $[0, 1]^n$, $\left\{\vecx \suchthat x_1 > 0, \ldots, x_n > 0\right\}$, or $\left\{\vecx \suchthat \|\vecx\|_2^2 < 1\right\}$. In more generality, we could have $\mathcal{U}$ be a semi-Pfaffian set defined by Pfaffian functions on a larger domain $\mathcal{U}' \supseteq \mathcal{U}$, which in turn is defined by Pfaffian functions on some $\mathcal{U}'' \supseteq \mathcal{U}'$, and so on (see \cite{gabrielov2003relative}).
\end{remark}

\begin{example}[Examples of Pfaffian functions] \label{example:pfaffian}
\begin{enumerate}[label=(\Roman*),ref=(\Roman*)]
\item \label{point:polynomial-example} A polynomial of degree $D$ is a Pfaffian function w.r.t. the empty chain; format is $(\alpha, D, 0)$ for any integer $\alpha>0$.
\item If $q(x) = e^{ax}$, $\vec{q} = (q(x))$ is a Pfaffian chain over $\R$ given $\de q(x) = a q(x) \de x$; $\vec{q}$ has order $1$ and chain-degree $1$. Any $P \in \R[X, e^{aX}]$ is Pfaffian function w.r.t. $\vec{q}$.
\item As a generalization of the example above, for $i \in [r]$, let $q_i(x) = e^{q_{i-1}(x)}$, and $q_0(x) = ax$. Since $\de q_i(x) = aq_1(x)\ldots q_i(x) \de x$, $\vec{q} = (q_1, \ldots, q_r)$ is a Pfaffian chain of order $r$ and chain-degree $r$. Consequently, any $P \in \R[X, e^{aX}, e^{e^{aX}}, \ldots ]$ is a Pfaffian function w.r.t. $\vec{q}$.
\item $\vec{q} = (\tan(x))$ is a Pfaffian chain of order $1$ and chain-degree $2$ in the domain $\bigcap_{k \in \Z} \{x \in \R \suchthat x \neq \nicefrac{\pi}{2} + k\pi\}$, given $\de \tan(x) = \left(1 + \tan(x)^2\right) \de x$; any $P \in \R[X, \tan(X)]$ is a Pfaffian function w.r.t. $\vec{q}$. 
\item $\vec{q} = \left(\frac{1}{x}, \ln(x)\right)$ is a Pfaffian chain on the domain $R \setminus \{0\}$ since \[\de \frac{1}{x} = -\left(\frac{1}{x}\right)^2 \de x, \qquad \text{and} \qquad \de \ln(x) = \frac{1}{x} \de x.\] Consequently, any $P \in \R\left[X, \frac{1}{X}, \ln(X)\right]$ is a Pfaffian function.
\item \label{point:x-to-m-example} $\vec{q} = \left(\frac{1}{x}, x^m\right)$ for any $m \in \R$ is a Pfaffian chain since \[
\de \frac{1}{x} = -\left(\frac{1}{x}\right)^2 \de x, \qquad \text{and} \qquad \de x^m = m \cdot \frac{1}{x} \cdot x^m \de x.
\] Thus any $P \in \R\left[X, \frac{1}{X}, X^m\right]$ is a Pfaffian function w.r.t. $\vec{q}$.
\item \label{point:fewnomials-example} For any monomial $m_{i_1, \ldots, i_n} := a_{i_1, \ldots, i_n}x_1^{i_1}\ldots x_n^{i_n}$ with $a_{i_1, \ldots, i_n} \neq 0$, $\left(\frac{1}{x_1}, \ldots, \frac{1}{x_n}, m_{i_1, \ldots, i_n}\right)$ is a Pfaffian chain on the domain $\{\vecx \in \R^n \suchthat x_1\dots x_n \neq 0\}$ of order $n+1$, and chain-degree  $2$, given that
\begin{equation*}
\de m_{i_1, \ldots, i_n} = m_{i_1, \ldots, i_n}\cdot \sum_{j=1}^n i_j \frac{1}{x_j} \de x_j, \qquad \text{and} \qquad \de \frac{1}{x_j} = -\left(\frac{1}{x_j}\right)^2 \de x_j.
\end{equation*}
From this we can deduce that if $f$ is a polynomial that is a sum of $s$ such monomials, called a \emph{fewnomial of sparsity $s$}, then $f$ is a Pfaffian function (on the same domain as before) of order $n+s$ and degree $(2, 1)$.
\end{enumerate}
\end{example}

We define the following restricted class of Pfaffian chains.

\begin{definition}[Algebraically independent Pfaffian chain]
\label{defn:algebraically-independent-pfaffian-chain}
A Pfaffian chain $\vec{q} = (q_1, \ldots, q_r)$ is called \emph{algebraically independent} if for all non-zero $P \in \R[Y_1, \ldots, Y_r]$, $P(q_1(\x), \ldots, q_r(\x))$ is not identically zero.
\end{definition}

\begin{remark}[Regarding algebraic independence of Pfaffian chains]
    We prove two types of partitioning theorems, namely partitioning using a polynomial and partitioning using a Pfaffian function, and we place the requirement that the common Pfaffian chain be algebraically independent only in the latter case. This is solely to ensure that the partitioning Pfaffian function obtained is non-zero; the zero function trivially satisfies the promised conditions of partitioning theorems, but is worthless from the point of view of applications. 

    How restrictive is the condition of algebraic independence? All examples in Example \ref{example:pfaffian}, except \ref{point:x-to-m-example} and \ref{point:fewnomials-example}, have algebraically independent Pfaffian chains. In fact, if in Example \ref{example:pfaffian}-\ref{point:x-to-m-example}, $m \in \R \setminus \Q$, then $\vec{q}$ would be algebraically independent (see below for a proof of this). Also, if the different elements of the Pfaffian chain contain different elementary functions, the Pfaffian chain will be algebraically independent. 
\end{remark}

\begin{proof}[Proof of algebraic independence of the Pfaffian chain in Example \ref{example:pfaffian}-\ref{point:x-to-m-example} when $m \in \R \setminus \Q$]
    We wish to prove that $\vec{q} = \left(\frac{1}{x}, x^m\right)$ is an algebraically independent Pfaffian chain whenever $m$ is irrational.

    Suppose, by way of contradiction, that for some fixed $m \in \R \setminus \Q$, there is a polynomial $P \in \R[Y_1, Y_2]$, of degree say $D$, such that $P \neq 0$ and $P\left(\frac{1}{X}, X^m\right)$ is identically $0$. Suppose that $P$ is made of monomials $\{Y_1^{a_i}Y_2^{b_i}\}$ such that $a_i, b_i \in \Z_{\ge 0}$ and $a_i + b_i \le D$. On substitution, the monomial $Y_1^{a_i}Y_2^{b_i}$ becomes $X^{mb_i - a_i}$. Notice that the map
    \[
    Y_1^{a_i}Y_2^{b_i} \mapsto X^{mb_i - a_i}
    \] is injective; because if not, it would mean that there exist $a, b, a', b' \in \Z_{\ge 0}$ such that $(a, b) \neq (a', b')$ and $mb' - a' = mb - a$. This in turn would mean that $m = \frac{a' - a}{b' - b}$, which is a contradiction to the assumption that $m \in \R \setminus \Q$. Consequently, if $P\left(\frac{1}{X}, X^m\right)$ is identically $0$, it means that $mb_i - a_i = 0$ for all $(a_i, b_i)$; which is also a contradiction to the assumption that $m \in \R \setminus \Q$.
\end{proof}

\begin{definition}[Pfaffian sets] A \emph{Pfaffian set} in an open domain $\mathcal{U} \subseteq \R^n$ is a set of the form \[Z(g_1, \ldots, g_k) := \{\vecx \in \mathcal{U} \suchthat g_1(\vecx) = \cdots = g_k(\vecx) = 0\},\] where $g_i: \R^n \rightarrow \R$ are Pfaffian functions.
\end{definition}

\begin{definition}[Semi-Pfaffian sets] A set $\Pfaff \subseteq \R^n$ is called \emph{semi-Pfaffian} in an open domain $\mathcal{U} \subseteq \R^n$ if it is the locus in $\mathcal{U}$ of a Boolean combination of equations $f = 0$ and inequalities $g > 0$, where $f, g$ are Pfaffian functions having a common Pfaffian chain defined on $\mathcal{U}$. 

A \emph{basic semi-Pfaffian set} is one where the corresponding Boolean combination is just a conjunction of equations and strict inequalities. That is, a basic semi-Pfaffian set is of the form
\begin{equation*}
  \Pfaff = \{\vecx \in \mathcal{U} \colon g_1(\vecx)=\cdots =g_s(\vecx)=0, \ h_1(\vecx)>0,\dots,h_t(\vecx)>0\},
\end{equation*}
where the $g_i$ and $h_j$ are Pfaffian functions having a common Pfaffian chain defined on $\mathcal{U}$.

The \emph{complexity} of a semi-Pfaffian set $\Pfaff$ is defined as the coordinatewise maximum of the formats of the Pfaffian functions that define $\Pfaff$. 

We say a semi-Pfaffian set has \emph{order} $r$ if it can be defined by Pfaffian functions defined w.r.t. the same Pfaffian chain of order not exceeding $r$.
\end{definition}

Our proofs depend on the fact that semi-Pfaffian sets can be decomposed into semi-Pfaffian subsets, in such a way that the complexity of these subsets can be controlled. Relevant definitions and results below are based on~\cite{gabrielov2004complexity}, and we refer to that source for more background, proofs and references.

\begin{definition}[Weak stratification]
\label{defn:weak-stratification}
A {\em weak stratification} of a semi-Pfaffian set $\Pfaff$ is a partition of $\Pfaff$
into a disjoint union of smooth, not necessarily connected, possibly empty semi-Pfaffian subsets $\Pfaff_i$, called {\em strata}.
A stratification is {\em basic} if all strata are basic semi-Pfaffian sets that are {\em effectively nonsingular}, i.e.,
the system of equations and strict inequalities for each stratum $\Pfaff_i$ of codimension $k$ includes a set of $k$
Pfaffian functions $h_{1}, \ldots , h_{k}$ vanishing on $\Pfaff_i$, such that
their gradients are linearly independent at every point in $\Pfaff_i$.
\end{definition}

\subsection{Pfaffian and Real algebraic geometry preliminaries}
In this section, we collect a few well-known results in Pfaffian and real-algebraic geometry that we use in the rest of the paper.

An important result in the theory of Pfaffian functions is Khovanski{\u\i}'s B\'ezout type theorem for Pfaffian sets. 

\begin{theorem}[Khovanski{\u\i}'s theorem {\cite[\S 3.12, Corollary 5]{khovanskiui1991fewnomials}}] \label{thm:khovanskii} Let $\vec{q}$ be a Pfaffian chain of order $r$ and chain-degree $\alpha$, where the functions in the Pfaffian chain depend only on $\xi \le n$ variables. Let $f_1, \ldots, f_n$ be Pfaffian functions on an open set $\mathcal{U} \subseteq \R^n$, where $f_i \in \Pf_n(\beta_i, \vec{q})$. The number of non-degenerate solutions of the system $\{\x \in \mathcal{U} \suchthat f_1(\x) = \cdots = f_n(\x) = 0\}$ is bounded from above by
\[2^{\binom{r}{2}}\beta_1 \cdots \beta_n\left(\min\{\xi, r\}\alpha + \beta_1 + \cdots + \beta_n - n + 1\right)^r.\]
\end{theorem}

From hereon, we will omit reference to the domain $\mathcal{U}$ and implicitly assume that all occurring Pfaffian functions are defined on the same domain.

The following proposition establishes the existence of a basic weak stratification for any semi-Pfaffian set.

\begin{proposition}[\cite{gabrielov2004complexity}*{Theorem 6.2, Smooth stratification}]
\label{prop:strat}
Let $\vec{q}$ be a Pfaffian chain of order $r$ and chain degree $\alpha$. Any semi-Pfaffian set $\Pfaff \subset \R^n$ defined by $m$ Pfaffian functions in
$\Pf_n(\beta, \vec{q})$ has a finite basic weak stratification. There exist constants $C_1=C_1(\alpha, r, m, n)$ and $C_2=C_2(r,n)$ such that the number of connected components of all strata and their complexities are bounded from above by $C_1\beta^{C_2}$.
\end{proposition}

Next, we state a couple of well-known results in real algebraic geometry. The first is a bound on the total number of connected components in $\R^n$ of the complement of the zero set of a polynomial $P \in \R[X_1, \ldots, X_n]$.

\begin{theorem}[\cites{petrovskii1949topology,milnor1964betti,thom1965homologie}]
\label{thm:optm} For any $P \in \R[X_1, \ldots, X_n]$ of degree $D$, there exists a constant $C = C(n)$ such that \[b_0(\R^n \setminus Z(P)) \le CD^n.\]
\end{theorem}

From Theorem \ref{thm:optm}, we have that the zero set of any $P \in \R[X_1, \ldots, X_n]$ breaks $\R^n$ into at most $\mathcal{O}_n(D^n)$ connected components. For $k < n$, given a $k$-dimensional real-algebraic set $\Pfaff$, the following theorem gives an upper bound on the number of connected components of $\R^n \setminus Z(P)$ that $\Pfaff$ intersects. We state the formulation of \cite[Theorem A.2]{solymosi2012incidence} which is implied by a more general result \cite[Theorem 4]{barone2016real}.

\begin{theorem}[{\cite[Theorem 4]{barone2016real}, \cite[Theorem A.2]{solymosi2012incidence}}]
\label{thm:barone-basu-ST}
Suppose that $\Pfaff \subseteq \R^n$ is a $k$-dimensional real-algebraic set defined by $m$ polynomial equations, each of degree at most $d$. Given any polynomial $P \in \R[X_1, \ldots, X_n]$ of degree at most $D$, then there is a constant $C = C(n, m, d)$ such that \[b_0\left(\Pfaff \cap (\R^n \setminus Z(P))\right) \le CD^k.\]
\end{theorem}

The key to our results is a generalization of the above bound to semi-Pfaffian sets.

\subsection{Auxiliary Tools}
\label{sec:preliminaries-tools}
In this Section, we prove a few auxiliary results that we shall use in proving our main theorems as well as applications of our main theorems. 

Since Theorem \ref{thm:khovanskii} only gives a bound on the number of non-degenerate solutions of a Pfaffian system, we need the following corollary.

\begin{corollary}[of Theorem \ref{thm:khovanskii}] \label{cor:khovanskii}
Let $\vec{q}$ be a Pfaffian chain of order $r$ and chain-degree $\alpha$, where the functions in the Pfaffian chain depend only on $\xi \le n$ variables. Let $f_1, \ldots, f_n$ be Pfaffian functions on an open set $\mathcal{U} \subseteq \R^n$, where $f_i \in \Pf_n(\beta_i, \vec{q})$. Suppose we have that for all $i \in [n]$, $\mathrm{dim}\, Z(f_i) = n-1$ and $\mathrm{dim}\, Z(f_1, \ldots, f_i) = n - i$. Then the number of solutions of the system $M:= \{\x \in \mathcal{U} \suchthat f_1(\x) = \cdots = f_k(\x) = f_{k+1}(\x) = \cdots = f_n(\x) = 0\}$ is bounded from above by
\[2^{n + \binom{r+1}{2}}\beta_1 \cdots \beta_n\left(\min\{\xi, r\}\alpha + \beta_1 + \cdots + \beta_n - n + 1\right)^r.\]
\end{corollary}

\begin{proof}[Proof of Corollary \ref{cor:khovanskii}]
First we introduce the following definition from \cite[Definition 1.7]{gabvorobapprox}. Let ${\mathcal P}= {\mathcal P}(\eps_1, \ldots ,\eps_\ell)$ be a predicate (property) over
$(0,1)^\ell$. We say that $\mathcal P$ {\em holds for} $0 < \eps_\ell \ll \eps_{\ell-1} \ll \cdots \ll \eps_1 \ll 1$ if there exist functions $t_k:\> (0,1)^{k-1} \to (0,1)$, definable in real analytic structure $\R_{an}$ (see for e.g. \cite[Section 2.1.2]{binyamininovikovicm}), where $k=1, \ldots, \ell$ (with $t_1$ being a positive constant) such that
$\mathcal P$ holds for any sequence $\eps_1, \ldots ,\eps_\ell$ satisfying
$0< \eps_k <t_k(\eps_{k-1}, \ldots ,\eps_1)$.

For real numbers $\eps_1, \ldots, \eps_n$ consider the set $Z(f_1^2- \eps_1, \ldots, f_n^2-\eps_n)$.
We prove, by induction on $\ell$, that for $0< \eps_\ell \ll \cdots \ll \eps_1 \ll 1$ all solutions
$x$ in $Z(f_1^2- \eps_1, \ldots, f_\ell^2-\eps_\ell)$ are regular, i.e., the matrix
$[\nabla f_1^2 ({\x}), \ldots ,\nabla f_\ell^2 ({\x})]$ has full rank, and tend to
solutions in $Z(f_1, \ldots , f_\ell)$ as $\eps_i \to 0$.

The base case of the induction, i.e. when $\ell=1$, is obvious.

Now, for $\ell < n$, assume that all points in $Z(f_1^2-\eps_1, \ldots ,f_\ell^2-\eps_{\ell})$ are regular.
Consider the restriction $g_{\ell+1}$ of $f_{\ell+1}^2$ to $Z(f_1^2-\eps_1, \ldots ,f_\ell^2-\eps_{\ell})$.
Note, that $Z(f_1^2-\eps_1, \ldots ,f_\ell^2-\eps_{\ell}, g_{\ell+1})$ is $(n-\ell-1)$-dimensional
and  the function $g_{\ell+1}$
attains the minimum ($ = 0$) on $Z(f_1^2-\eps_1, \ldots ,f_\ell^2-\eps_{\ell})$ at its critical points.
Thus, $g_{\ell+1}-\eps_{\ell+1}$, for small $\eps_{\ell+1}>0$, does not have critical points on
$Z(f_1^2-\eps_1, \ldots ,f_\ell^2-\eps_{\ell})$ with value $0$, hence all points in
$Z(f_1^2-\eps_1, \ldots ,f_\ell^2-\eps_{\ell}, f_{\ell+1}^2-\eps_{\ell + 1})$  are regular.

Finally, applying Theorem \ref{thm:khovanskii} to $Z(f_1^2-\eps_1, \ldots , f_n^2- \eps_n)$, noticing that the degrees of $f_i^2 - \eps_i$ are $2\beta_i$, completes the proof.
\end{proof}

\begin{notation}
    For any semi-Pfaffian set $\Pfaff \subset \R^n$, let $\smooth{\Pfaff}$ be the subset of all points $x \in \Pfaff$ such that the neighbourhood of $x$ in $\Pfaff$ is a smooth manifold.
\end{notation}

Next, Corollary \ref{cor:smooth-decomposition} and Corollary \ref{cor:irreducible-components} below use the stratification result in Proposition \ref{prop:strat} to decompose a Pfaffian set into smooth sets, and also provides some quantitative results on the topology of Pfaffian sets.

\begin{corollary}[of Proposition \ref{prop:strat}]
\label{cor:smooth-decomposition}
Let $\vec{q}$ be a Pfaffian chain of chain-degree $\alpha$ and order $r$. Let $\Pfaff \subseteq \R^n$ be a $k$-dimensional Pfaffian set defined by at most $m$ Pfaffian functions from $\Pf_n(\beta, \vec{q})$. Then there exist a constant $C = C(n, m, \alpha, \beta, r)$ such that one can cover $\Pfaff$ as
\[\Pfaff = \bigcup_{i=1}^{C} \smooth{\Pfaff_i},\]
where each $\Pfaff_i \subseteq \Pfaff$ is a Pfaffian set of dimension at most $k$ defined by at most $C$ Pfaffian functions
from $\Pf_n(C, \vec{q})$.
\end{corollary}

\begin{proof}
Using Proposition \ref{prop:strat}, perform a basic weak stratification of $\Pfaff$ to obtain strata $S_1, \ldots, S_C$. Each stratum $S_i$ is represented as a set of points in $\R^n$ satisfying a system of equations $E_{i}$, and strict inequalities $I_{i}$. Define $\Pfaff_i$ to be the intersection of the set defined by $E_{i}$ and $\Pfaff$. Then we claim that $\Pfaff_i$ are Pfaffian sets, and that the $\smooth{\Pfaff_i}$ form the required covering of $\Pfaff$. 

Indeed, for all $i \in [C]$, $\Pfaff_i \subseteq \Pfaff$, so $\bigcup_{i=1}^C \smooth{X_i} \subseteq \Pfaff$. Also, for any $x \in \Pfaff$, there exists $S_i$ such that $x \in S_i$. Since $S_i \subseteq \smooth{\Pfaff_i} \subseteq \Pfaff$, we have that $x \in \smooth{\Pfaff_i} \subseteq \bigcup_{i=1}^C \smooth{\Pfaff_i}$, proving that $\Pfaff \subseteq \bigcup_{i=1}^C \smooth{\Pfaff_i}$, and consequently, we have \[ \Pfaff = \bigcup_{i=1}^C \smooth{\Pfaff_i},\] as required. 

Also, it is true that $\mathrm{dim} \, \Pfaff_i \le k$, and the conditions required on the complexity of $\Pfaff_i$ follow from the guarantee of Proposition \ref{prop:strat}.
\end{proof}

\begin{corollary}[of Proposition \ref{prop:strat}]
\label{cor:irreducible-components}
Let $\vec{q}$ be a Pfaffian chain of chain degree $\alpha$ and order $r$, and let $P \in \Pf_n(\beta, \vec{q})$. There exist constants $C_1=C_1(\alpha, \beta, r, n)$ and $C_2=C_2(r,n)$ such that the number of irreducible (over $\R$) components of $Z(P)$ is at most $C_1 \beta^{C_2}$.
\end{corollary}

\begin{proof}
Applying Proposition \ref{prop:strat} to $Z(P)$, we conclude that the number of connected components of smooth strata
in the stratification of $Z(P)$ does not exceed $C_1 \beta^{C_2}$, where both $C_1$ and $C_2$ are the constants coming from Proposition \ref{prop:strat} itself.

Let $Z(Q)$ be an irreducible component of $Z(P)$ having dimension $d$. We have that there is a connected component $Y$ of a $d$-dimensional stratum such that $\dim (Z(Q) \cap Y) = d$, i.e., the function $Q$ vanishes on a $d$-dimensional disc in $Y$. By the identity theorem (see for instance \cite[Section 1.2]{krantz2002primer}), $Q$ vanishes on all of $Y$, hence $Y \subseteq Z(Q)$. Thus, it follows that the number of irreducible components of $Z(P)$ is at most the number of connected components of smooth strata of $Z(P)$. Consequently, $C_1\beta^{C_2}$ is an upper bound on the number of irreducible components.
\end{proof}

\section{Polynomial and Pfaffian Partitioning of a Set of Semi-Pfaffian Sets}
In this section, we prove our main theorem, a partitioning theorem for a set of semi-Pfaffian sets (Theorem \ref{thm:poly-part-guth}).

\subsection{Smooth decomposition of semi-Pfaffian sets}
The first step towards Theorem \ref{thm:poly-part-guth} is an upper bound on the number of connected components of $\Pfaff \cap (\R^n \setminus Z(P))$ for a semi-Pfaffian set $\Pfaff$ and a polynomial $P$ (Theorem \ref{thm:bezout-pfaffian}). The following lemma describes a procedure to obtain a smooth decomposition of semi-Pfaffian sets using the weak smooth stratification result stated in Proposition \ref{prop:strat}.

\begin{lemma}
\label{lem:part}

Let $\mu>0$ and let $\Pfaff \subseteq (-\mu, \mu)^n \subseteq \R^n$ be a semi-Pfaffian set of order $r$. Then there exists a partition of $\Pfaff$ into semi-Pfaffian subsets, called cells, such that each $n$-dimensional cell is an open set in $(-\mu, \mu)^n$, while each
$k$-dimensional set for $k < n$ is the graph of a smooth map ${\bf f}=(f_1, \ldots ,f_{n-k})$ defined on an open set in one of the $k$-dimensional faces of $(-\mu, \mu)^n$. Moreover, each cell has order $r$, and is \emph{effectively non-singular} (see Definition \ref{defn:weak-stratification}). The number of cells, the number of functions defining cells (open sets and graphs), and the complexities of cells are bounded from above by an explicit function of the complexity of $\Pfaff$.
\end{lemma}

\begin{proof}
The proof is by induction on $\dim \Pfaff=k$. If $k=0$, then $\Pfaff$ consists of a finite number of points, and each of these constitutes a cell that satisfies the required properties (the bound on the number of points is provided by Theorem~\ref{thm:khovanskii}). Assume $k>0$.
By Proposition \ref{prop:strat}, there exists a weak, smooth stratification of $\Pfaff$ into a finite number of basic semi-Pfaffian sets. Let $\Pfaff_{<k}$ be the union of all strata of dimension less than $k$. 

Consider a stratum $\mathcal{C}$ with $\dim \mathcal{C}=k$. We first show that we can subdivide $\mathcal{C}$ into strata that can be represented as graphs, and later argue that the complexity of the functions defining the graphs can be controlled in terms of the parameters defining the semi-Pfaffian set. Since the stratification is basic,  
there exist Pfaffian functions $h_1,\dots,h_{n-k}$ such that $h_i|_{\mathcal{C}}=0$ and $\mathrm{d}h_1\wedge \cdots \wedge \mathrm{d}h_{n-k}\neq 0$ on $\mathcal{C}$.
For $\{ i_1, \ldots, i_m \} \subset [n]$, define the minor
\begin{equation*}
{\bf J}(i_1, \ldots, i_m ):=
\left( \begin{array}{ccc}
\partial h_1/\partial x_{i_1} & \cdots & \partial h_1 / \partial x_{i_m}\\
\vdots & \ddots & \vdots \\
\partial h_{n-k}/ \partial x_{i_1} & \cdots & \partial h_{n-k}/ \partial x_{i_m}
\end{array} \right).
\end{equation*}
Note that ${\bf J}(1,2, \ldots, n )$ has rank $n-k$ on $\mathcal{C}$. 

Fix some total order (for example, lexicographic) on the family of all subsets of $\{1, 2, \ldots ,n \}$ of cardinality $n-k$.
Suppose, for definiteness, that the minimal set, with respect to this order, is $\{ 1, 2, \ldots , n-k\}$.

Define
\begin{equation*}
\mathcal{C}_{1, \ldots, n-k} := \mathcal{C} \cap \{ \det {\bf J}(1, \ldots, n-k)=0 \}.
\end{equation*}
Observe that $\mathcal{C} \setminus \mathcal{C}_{1, \ldots, n-k}$ is either empty or $k$-dimensional, while $\dim \mathcal{C}_{1, \ldots, n-k} \le k$
(possibly empty).
If $\dim \mathcal{C}_{1, \ldots, n-k} < k$, then reset $\Pfaff_{<k}:= \Pfaff_{<k} \cup \mathcal{C}_{1, \ldots, n-k}$.
Otherwise, perform a weak stratification of $\mathcal{C}_{1, \ldots, n-k}$, reset $\Pfaff_{<k}$ by taking the union of current $\Pfaff_{<k}$ with strata of dimension less than $k$, and denote by $\widehat{\mathcal{C}}_{1, \ldots, n-k}$ the $k$-dimensional stratum.

Assume that we constructed the smooth $k$-dimensional $\widehat{\mathcal{C}}_{j_1, \ldots, j_{n-k}}$ for a set $\{ j_1, \ldots ,j_{n-k} \}$
in the chosen ordering of subsets of $\{1, 2, \ldots ,n \}$.
Let $\{ i_1, \ldots i_{n-k} \}$ be the next set in the ordering.
Define
\begin{equation*}
\mathcal{C}_{i_1, \ldots, i_{n-k}} := \widehat{\mathcal{C}}_{j_1, \ldots, j_{n-k}} \cap \{ \det {\bf J}(i_1, \ldots, i_{n-k})=0 \}.
\end{equation*}
Observe that $\widehat{\mathcal{C}}_{j_1, \ldots, j_{n-k}} \setminus \mathcal{C}_{i_1, \ldots, i_{n-k}}$ is either empty or $k$-dimensional,
while $\dim \mathcal{C}_{i_1, \ldots, i_{n-k}} \le k$ (possibly empty).

If $\dim \mathcal{C}_{i_1, \ldots, i_{n-k}} < k$, then reset $\Pfaff_{<k}:= \Pfaff_{<k} \cup \mathcal{C}_{i_1, \ldots, i_{n-k}}$.
Otherwise, perform a weak stratification of $\mathcal{C}_{i_1, \ldots, i_{n-k}}$, reset $\Pfaff_{<k}$ by taking the union of current $\Pfaff_{<k}$
with strata of dimension less than $k$, and denote by $\widehat{\mathcal{C}}_{i_1, \ldots, i_{n-k}}$ the $k$-dimensional stratum.

Continue this process by passing from one subset in the chosen ordering of subsets of $\{1, 2, \ldots ,n \}$ to the next.
The process will terminate if either the current $\mathcal{C}_{i_1, \ldots, i_{n-k}}$ has dimension less than $k$, or is empty.
The latter case occurs when for the previous set $\{ j_1, \ldots , j_{n-k} \}$ the stratum $\widehat{\mathcal{C}}_{j_1, \ldots, j_{n-k}}$
does not have points with vanishing $\det {\bf J}(i_1, \ldots, i_{n-k})$.
The process always terminates since effective non-singularity of $\mathcal{C}$ implies that the rank of the $(n-k) \times n$ - matrix
${\bf J}(1, \ldots n)$ is maximal at every point of $\mathcal{C}$.

By the implicit function theorem, each non-empty set $\widehat{\mathcal{C}}_{j_1, \ldots, j_{n-k}} \setminus \mathcal{C}_{i_1, \ldots, i_{n-k}}$
is a disjoint union of graphs of smooth maps defined on their projections to the $k$-dimensional face of $(-\mu, \mu)^n$ in the subspace
of coordinates $x_{\ell_1}, \ldots , x_{\ell_k}$, where
\[\{ \ell_1, \ldots , \ell_k \}= \{ 1, \ldots n \} \setminus \{ i_1, \ldots , i_{n-k} \}.\]
Hence, in the terminating ordered sequence of subsets of $\{1, 2, \ldots ,n \}$ the sets
\begin{equation*}
(\mathcal{C} \setminus \mathcal{C}_{1, \ldots, n-k}), \ldots , (\widehat{\mathcal{C}}_{j_1, \ldots, j_{n-k}} \setminus \mathcal{C}_{i_1, \ldots, i_{n-k}})
\end{equation*}
are disjoint unions of $k$-dimensional graphs of smooth maps, while the complement of their union
in $\mathcal{C}$ has a dimension less than $k$. Applying the inductive hypothesis to $\Pfaff_{<k}$, we obtain a partition of $\Pfaff$ into graphs of smooth maps.

In order to see that the functions $f_i$ provided by the implicit function theorem are Pfaffian functions whose parameters can be bounded in terms of the equations $h_1,\dots,h_k$ defining the original strata, note that from the implicit function theorem we can extract an expression for the partial derivatives $\partial f_i/\partial x_j$ in terms of the partial derivatives of the $h_i$. Complexities of functions $h_i$ are bounded from above by Proposition \ref{prop:strat}.
\end{proof}

\subsection{Proofs of Theorem \ref{thm:bezout-pfaffian} and Theorem \ref{thm:poly-part-guth}}
\label{sec:main-proofs-subsection}

We now have all the ingredients required to prove Theorem \ref{thm:bezout-pfaffian} and Theorem \ref{thm:poly-part-guth}.

\begin{proof}[Proof of Theorem \ref{thm:bezout-pfaffian}]
For all sufficiently large integers $\mu>0$, the number of connected components of $\Pfaff \cap (\R^n \setminus Z(P))$ does not exceed (in fact, coincides with) the number of connected components of $\Pfaff \cap (-\mu, \mu)^n \cap (\R^n \setminus Z(P))$. Therefore, without loss of generality, we will assume in this proof that $\Pfaff \subseteq (-\mu, \mu)^n$ for some $\mu>0$. Apply Lemma \ref{lem:part} to get a partition $\mathcal{C} = \{\mathcal{C}_1, \ldots, \mathcal{C}_v\}$ of $\Pfaff$. We will establish that there is a constant $C' = C'(n, k, m, \alpha, \beta, r)$ such that, for any $\mathcal{C}_i \in \mathcal{C}$,  \begin{equation} \label{eqn:to-prove} b_0(\mathcal{C}_i \setminus Z(P)) \le C'D^{k},\end{equation} thus completing the proof because, by Lemma \ref{lem:part}, we have that $|\mathcal{C}|$ does not grow with $D$.

We proceed by induction on $k$. The base case of $k = 0$ is trivial, so suppose that $k \ge 1$, and that \eqref{eqn:to-prove} is proved for all $\Pfaff$ of dimension strictly less than $k$. Fix a $k$-dimensional smooth cell $\mathcal{C}_i$. For this specific $\mathcal{C}_i$, let $\alpha_1, \ldots, \alpha_{t}$ be the connected components of $\{\vecx \in \mathcal{C}_i \suchthat P(\vecx) \neq 0\}$. Obviously this means that \begin{equation} \label{eqn:target-m} b_0(\mathcal{C}_i \setminus Z(P)) \le t.\end{equation} We might as well assume that all $\alpha_i$ do not intersect the boundary $\partial \mathcal{C}_i$, because the number of components that intersect $\partial \mathcal{C}_i$ is at most $\mathcal{O}_{n, \partial \mathcal{C}_i}(D^{k-1})$ by the induction hypothesis. Naturally by the continuity of $P$, on any particular $\alpha_j$, we will either have $\forall \vecx \in \alpha_j$, $P(\vecx) > 0$, or, $\forall \vecx \in \alpha_j$, $P(\vecx) < 0$. Consequently, $P$ is guanteed to have at least one critical point on each $\alpha_j$.

Define $A := \{i \in [t] \suchthat \forall \vecx \in \alpha_i, P(\vecx) > 0 \}$. Choose $\eps > 0$ to be a regular value of $P$ that is strictly smaller than the minimum of all the critical values of $P$ on all $\alpha_j$, $j \in A$, which are all obviously positive because of how $A$ is defined. Since $\eps$ is positive and is also strictly smaller than the minimum critical value of $P$, for all $j \in A$, $Z(P - \eps) \cap \alpha_j \neq \emptyset$. Also, since $\eps$ is a regular value of $P$, by the implicit function theorem, $Z(P - \eps)$ is smooth. Given that by Lemma \ref{lem:part}, all $\alpha_j$, $j \in [t]$, are also smooth, we can conclude that $Z(P - \eps) \cap \alpha_j$ is smooth. $Z(P - \eps) \cap \alpha_j$ is also guaranteed to be compact because $Z(P - \eps)$ cannot touch $\partial \alpha_j$ for all $j \in A$. Our goal will be to find an upper bound on $|A|$. Notice that $[t] \setminus A$ is exactly the set $\{i \in [t] \suchthat \forall \vecx \in \alpha_i, P(\vecx) < 0 \}$. Thus by analysing similarly as what is above with $\eps < 0$, and looking at $P + \eps$ instead of $P - \eps$, we will be able to obtain an upper bound on the cardinality of the complement of $A$ in $[t]$; summing up both upper bounds will give an upper bound on $t$.

By the definition of the partition in Lemma~\ref{lem:part}, $\mathcal{C}_i$ is the graph of a smooth map ${\bf f}_i=(f_{i,1}, \ldots f_{i, n-k})$
defined on the projection of $\mathcal{C}_i$ on a $k$-dimensional face $F$ of $(-\mu, \mu)^n$.
Assume, for definiteness, that $F$ has coordinates $x_1, \ldots , x_k$.
Then, $Z(P - \eps) \cap \mathcal{C}_i$ is $(k-1)$-dimensional smooth Pfaffian set, and it's projection
${\rm proj}\ (Z(P - \eps) \cap \mathcal{C}_i)$ on $F$ is a {\em smooth embedding}.
In particular, ${\rm proj}\ (Z(P - \eps) \cap \mathcal{C}_i)$ is a smooth compact manifold.

Consider the subset \[G:=Z\left(P- \eps, \nicefrac{\partial P}{\partial x_2}, \ldots , \nicefrac{\partial P}{\partial x_k}\right)\] of $Z(P- \eps)$, i.e., the set of points in $Z(P - \eps)$ at which the gradient of $P$ is collinear to $x_1$-axis. At any point ${\bf x} \in G \cap \mathcal{C}_i$ the gradient of $P$ projects injectively on the line, normal to
$\mathrm{proj}(Z(P - \eps) \cap \mathcal{C}_i)$ at ${\rm proj} ({\bf x})$.
It may happen, that ${\rm proj}\ (G \cap \mathcal{C}_i)$ has infinite number of points (in case when this set
consists of critical points of the Gauss map of the manifold ${\rm proj}\ (Z(P - \eps) \cap \mathcal{C}_i)$).
Sard's theorem implies that using a generic rotation of coordinates $x_1, \ldots , x_k$ we can
guarantee that ${\rm proj}\ G  \cap \mathcal{C}_i$ contains only regular points of the Gauss map,
and therefore, is finite.
As a result, the pre-image of this set under the embedding, $G \cap \mathcal{C}_i$, is also finite.

The set $G \cap \mathcal{C}_i$ is defined by the system of equations
\[\left\{\text{graph of }{\bf f}_i, P - \eps =  \frac{\partial P}{\partial x_2} = \ldots = \frac{\partial P}{\partial x_{k}} = 0\right\}
.\]
Recall from Lemma \ref{lem:part} that the graph of ${\bf f}_i$, which is just $\mathcal{C}_i$, is an intersection of an open set, say $U_i$, and a set of solutions of a system of Pfaffian equations. The domain of these Pfaffian equations must obviously be larger than $U_i$. Thus we use Corollary \ref{cor:khovanskii} to obtain an upper bound on $|G \cap \mathcal{C}_i|$; by doing so, we are able to conclude that the number of solutions of this system is at most $C''D^{k+r}$, for some constant $C'' = C''(n, k, m, \alpha, \beta, r)$ that comes from applying Corollary \ref{cor:khovanskii}.

As argued above, $Z(P - \eps) \cap \alpha_j$, for all $j \in [A]$, is compact; in other words, each connected compoment in $\{\alpha_j\}_{j \in A}$ properly contains at least one component of $Z(P - \eps)$. Thus at least one point of $G \cap \mathcal{C}_i$ is sure to be present in each of $\{\alpha_i\}_{i \in A}$. This means that \[|A| \le C''D^{k+r}.\] The quantitative results above would be exactly the same if we looked at $P + \eps$, for some $\eps < 0$, instead of $P - \eps$, so $C''D^{k+r}$ is an upper bound on $\left|[t] \setminus A\right|$ as well. This means that $2C''D^{k+r}$ is an upper bound on $t$, and thus by \eqref{eqn:target-m}, the proof is complete by choosing $C'$ appropriately.
\end{proof}



Using Theorem \ref{thm:bezout-pfaffian}, we are able to prove Theorem \ref{thm:poly-part-guth}. In fact, we are able to prove Theorem \ref{thm:poly-part-bbz}, which is a generalization of Theorem \ref{thm:poly-part-guth}.

\begin{theorem}[Polynomial partitioning of collections of semi-Pfaffian sets - generalization of Theorem \ref{thm:poly-part-guth}-\ref{point:normal-poly-part}]
\label{thm:poly-part-bbz}
Let $m \ge 1$ be an integer, and for each $i \in [m]$, let $\Gamma_i$ be a finite set of $k_i$-dimensional semi-Pfaffian sets all in $\R^n$, where each $\Pfaff \in \Gamma_i$ is defined by at most $h_i$ Pfaffian functions each of degree at most $\beta_i$, chain-degree at most $\alpha_i$, and order at most $r_i$. Then for any $D \ge 1$, there exists a non-zero polynomial $P \in \R[X_1, \ldots, X_n]$ of degree at most $D$, and constants $C_i = C_i(n, k_i, h_i, \alpha_i, \beta_i, r_i)$ such that 
\begin{enumerate}[label=(\Roman*),ref=(\Roman*)]
    \item \label{pfaffpp-collection:dim-at-least-1} for all $i \in [m]$ with $k_i > 0$, each connected component of $\R^n \setminus Z(P)$ intersects at most $C_im\frac{|\Gamma_i|}{D^{n-k_i-r_i}}$ elements of $\Gamma_i$, and
    \item  \label{pfaffpp-collection:dim-0} for all $i \in [m]$ with $k_i = 0$, each connected component of $\R^n \setminus Z(P)$ intersects at most $C_im\frac{|\Gamma_i|}{D^n}$ elements of $\Gamma_i$.
\end{enumerate}
\end{theorem}

Theorem \ref{thm:poly-part-bbz} can be proved by following the exact same strategy as in \cite{blagojevic2017polynomial}, but by using Theorem \ref{thm:bezout-pfaffian} at the appropriate place. We reproduce the main part of the argument in \cite{blagojevic2017polynomial} in Appendix \ref{appendix:poly-part-bbz} to make our document self-contained.

\begin{proof}[Proof of Theorem~\ref{thm:poly-part-guth}]
Part \ref{point:normal-poly-part} is just the special case of Theorem~\ref{thm:poly-part-bbz} with $m=1$.

For Part \ref{point:pfaffian-poly-part}, recall that we assume that all the Pfaffian functions involved are defined with respect to the same Pfaffian chain $(q_1, \ldots, q_r)$. To every $\Pfaff\in \Gamma$ we associate a semi-Pfaffian set $\Pfaff'\subset \R^{n+r}$ as follows. Let $\Pfaff$ be defined by Pfaffian functions $P_i(\x) = Q_i(\x, q_1(\x), \ldots, q_r(\x))$, $i \in [m]$. To $\Pfaff$, we associate $\Pfaff'\subset \R^{n+r}$, where $\Pfaff'$ is the set defined by the polynomials $Q_i\in \R[X_1, \ldots, X_n, Y_1, \ldots, Y_r]$ intersected with $Z(y_1-q_1(\x), \ldots, y_r-q_r(\x))$ (that is, the intersection of a semi-algebraic set with the graph of the Pfaffian chain). In other words, letting $\mathcal{G} \subseteq \R^{n+r}$ denote the graph of the map 
\[
\R^n \ni \x \mapsto (q_1(\x), \ldots, q_r(\x)) \in \R^{r}, 
\]
it is clear that $\Pfaff'$ is obtained by restricting the domain of $\mathcal{G}$ to $\Pfaff$. 

Let $\Gamma'$ denote the collection of the associated semi-algebraic sets $\Pfaff'$. Each $\Pfaff'$ has dimension $k$, so we can apply Part \ref{point:normal-poly-part} to infer the existence of a polynomial $P\in \R[X_1, \ldots, X_n, Y_1, \ldots, Y_r]$ of degree at most $D$ such that each connected component of $\R^{n+r}\setminus Z(P)$ intersects at most $C'D^{k+r-(n-r)}|\Gamma'|=C'D^{k-n}|\Gamma|$ elements of $\Gamma'$. Specifically, since all $\Gamma' \ni \Pfaff' \subseteq \mathcal{G}$, the guarantee of Part \ref{point:normal-poly-part} is equivalent to saying that each connected component of $\mathcal{G} \setminus (\mathcal{G} \cap Z(P))$ contains at most $C'D^{k-n}|\Gamma|$ elements of $\Gamma'$. 

Define $P'(\x) := P(\x, q_1(\x), \ldots, q_r(\x)) \in \Pf_n(\alpha, \beta, r)$. $\mathcal{G}$ is homotopic to an open set in $\R^n$, thus the projection of $\mathcal{G}$ onto $\R^n$ carries each $\Pfaff'$ to $\Pfaff$, and also carries $\mathcal{G} \cap Z(P)$ to $Z(P')$. In other words, intersections of the various $\Pfaff \in \Gamma$ with connected components of $\R^n\setminus Z(P')$ bijectively correspond to intersections of the corresponding $\Pfaff' \in \Gamma'$ with connected components of $\mathcal/{G} \setminus (\mathcal{G} \cap Z(P))$. Since each connected component of $\mathcal{G} \setminus (\mathcal{G} \cap Z(P))$ contains at most $C'D^{k-n}|\Gamma|$ elements of $\Gamma'$, the bijectivity of the projection from $\mathcal{G}$ onto $\R^n$ ensures that each connected component of $\R^n \setminus Z(P')$ contains at most $C'D^{k-n}|\Gamma|$ elements of $\Gamma$.

It remains to prove that $P'$ is non-zero. By the guarantee of Theorem \ref{thm:poly-part-guth}-\ref{point:normal-poly-part}, $P$ is non-zero. Since $P \in \R[X_1, \ldots, X_n, Y_1, \ldots, Y_r]$, we can also consider that $P \in \R[Y_1, \ldots, Y_r][X_1, \ldots, X_n]$, i.e. that $P$ consists of monomials in the variables $X_1, \ldots X_n$ with coefficients from $\R[Y_1, \ldots, Y_r]$. Since $P$ is non-zero, at least one of the coefficients is non-zero, and since we have assumed that $\vec{q}$ is an algebraically independent chain, we have that no coefficients of P will become identically zero when we substitute all $Y_i$ by $q_i$. Thus we have that $P'$ will have at least one non-zero coefficient, proving that $P'$ is non-zero.
\end{proof}

By following the exact same strategy as in the proof of Theorem \ref{thm:poly-part-guth} - \ref{point:pfaffian-poly-part} and by using Theorem \ref{thm:poly-part-bbz} instead of Theorem \ref{thm:poly-part-guth} - \ref{point:normal-poly-part}, we get the following theorem, whose proof we only sketch.

\begin{theorem}[Pfaffian partitioning of collections of semi-Pfaffian sets, - generalization of Theorem \ref{thm:poly-part-guth}-\ref{point:pfaffian-poly-part}]
\label{thm:pfaff-part-bbz}
Let $\vec{q} = (q_1, \ldots, q_r)$ be an algebraically independent Pfaffian chain of order $r$ and chain-degree $\alpha$. Let $m \ge 1$ be an integer, and for each $i \in [m]$, let $\Gamma_i$ be a finite set of $k_i$-dimensional semi-Pfaffian sets all in $\R^n$, where each $\Pfaff \in \Gamma_i$ is defined by at most $h_i$ functions from $\Pf_n(\beta_i, \vec{q})$.  For any $D \ge 1$, there is a non-zero Pfaffian function $P' \in \Pf_n(D, \vec{q})$, and constants $C_i = C_i(n, k_i, h_i, \alpha, \beta_i, r)$ such that 
\begin{enumerate}[label=(\Roman*),ref=(\Roman*)]
    \item for all $i \in [m]$ with $k_i > 0$, each connected component of $\R^n \setminus Z(P')$ intersects at most $C_im\frac{|\Gamma_i|}{D^{n - k_i}}$ elements of $\Gamma_i$, and
    \item for all $i \in [m]$ with $k_i = 0$, each connected component of $\R^n \setminus Z(P')$ intersects at most $C_im\frac{|\Gamma_i|}{D^{n+r}}$ elements of $\Gamma_i$.
\end{enumerate}
\end{theorem}

\begin{proof}[Proof sketch of Theorem \ref{thm:pfaff-part-bbz}]
Just as in the proof of Theorem \ref{thm:poly-part-guth}-\ref{point:pfaffian-poly-part}, we lift from $\R^n$ each collection of semi-Pfaffian sets $\Gamma_i$ to $\R^{n+r}$ to form $\{\Gamma_i'\}_{i \in [m]}$ where each $\Gamma_i'$ is a collection of $k_i$-dimensional semi-Pfaffian sets in $\R^{n+r}$. We now apply Theorem \ref{thm:poly-part-bbz} to find a polynomial $P \in \R[X_1, \ldots, X_n, Y_1, \ldots, Y_r]$ of degree at most $D$ such that 
\begin{enumerate}[label=(\Roman*),ref=(\Roman*)]
    \item for all $i \in [m]$ with $k_i > 0$, each connected component of $\R^{n+r} \setminus Z(P)$ intersects at most $\sim \frac{|\Gamma_i'|}{D^{n-k_i}}$ elements of $\Gamma_i'$, and
    \item  for all $i \in [m]$ with $k_i = 0$, each connected component of $\R^{n+r} \setminus Z(P)$ intersects at most $\sim \frac{|\Gamma_i'|}{D^{n+r}}$ elements of $\Gamma_i'$.
\end{enumerate}

Next, again as in the proof of Theorem \ref{thm:poly-part-guth}-\ref{point:pfaffian-poly-part}, we turn $P$ into a Pfaffian function 
\[
P'(\vecx) = P(\vecx, q_1(\vecx), \ldots, q_r(\vecx)) \in \Pf_n(D, \vec{q}),
\]
and $P'$ satisfies the required guarantees on the collection $\{\Gamma_i\}_{i \in [m]}$.
\end{proof}

\newcommand{\sztr}{Szemer\'edi-Trotter}

\section{Applications}
\label{sec:applications}

In this section, we derive some applications of our main theorems. Section \ref{sec:applications-theoremstatements-sztr} and Section \ref{sec:applications-theoremstatements-joints} contain definitions and theorem statements, Section \ref{sec:applications-discussion} contains some discussion on finer technical details with regard to the results, and Section \ref{sec:applications-proofs} contains proofs of the stated theorems.

\subsection{Szemer\'edi-Trotter Type Theorems}
\label{sec:applications-theoremstatements-sztr}

Given a set $\PP$ of points in $\R^n$, and a set $\Gamma$ of subsets of $\R^n$, the set of \emph{incidences} between $\PP$ and $\Gamma$ is defined as 
\begin{equation*}
    I(\PP, \Gamma) := \left\{(p, \gamma) \in \PP \times \Gamma \suchthat p \in \gamma\right\}.
\end{equation*}
The classical Szemer\'edi-Trotter Theorem~\cite{szemereditrotter} proves a bound of $O(|\PP|^{2/3}|\Gamma|^{2/3}+|\PP|+|\Gamma|)$ on the number of incidences when $|\Gamma| \subseteq \R^2$ consists purely of lines. Under certain conditions on the incidence graph (see~\cites{pachsharir1998pointscurves,clarkson1990combinatorial} or~\cite[Chapter 3]{sheffer2022polynomial}), this bound has been generalized to the case of algebraic curves. In \cite{solymosi2012incidence}, incidences between points and higher dimensional varieties are studied under specific non-degeneracy assumptions; we refer to \cite{sheffer2022polynomial} for many more such theorems when $\Gamma$ contains algebraic sets. Also, as mentioned earlier, \cites{basu2018minimal,chernikov2020cutting} derive o-minimal/distal Szemer\'edi-Trotter type theorems.

The \emph{incidence graph} of $\Gamma$ and $\PP$, denoted $G_{\PP, \Gamma}$, is a bipartite graph $G_{\PP, \Gamma} = (V_{\PP} \cup V_{\Gamma}, E)$, where the vertices $V_{\PP}$ corresponding to points in $\PP$ and the vertices $V_{\Gamma}$ correspond to elements of $\Gamma$, and $(v_p, v_{\gamma}) \in E$ if $p \in \gamma$. 
Let $K_{s,t}$ denote the complete bipartite graph with $s$ vertices on one side and $t$ vertices on the other. The following Theorem counts incidences between points and Pfaffian curves in $\R^n$ under some non-degeneracy assumptions.

\begin{theorem}
\label{thm:st-rn}
Let $m, \alpha, r, \beta$, $n \ge 3$, and $s, t \ge 2$, be integers, and let $\eps > 0$. Then there exist constants $c_j = c_j(\eps, n, m, \alpha, \beta, r, s, t)$, for $1 < j < n$, such that the following holds. Let $\vec{q}$ be an algebraically independent Pfaffian chain of order $r$ and chain-degree $\alpha$. Let $\PP \subseteq \R^n$ be a set of points, and let $\Gamma \subseteq \R^n$ be a set of distinct irreducible Pfaffian curves such that each $\gamma \in \Gamma$ is defined by at most $m$ Pfaffian functions from $\Pf_n(\beta, \vec{q})$. Suppose that $G_{\PP, \Gamma}$ does not contain any copy of $K_{s,t}$. Also assume that, for all $1 < j < n$, every $j$-dimensional irreducible Pfaffian set defined by Pfaffian functions from $\Pf_n(c_j, \vec{q})$ contains at most $\theta_j$ curves of $\Gamma$. Then there exist constants $C_1 = C_1(\eps, n, m, \alpha, \beta, r, s, t)$ and $C_2 = C_2(\eps, n, m, \alpha, \beta, r, s, t)$ such that
\[
|I(\PP, \Gamma)| \le C_1 |\PP|^{\frac{s(r+1)}{s(n+r)-n+1} + \eps}|\Gamma|^{\frac{(s-1)(n+r)}{s(n+r) -n + 1}} + |\PP|\sum_{j=2}^{n-1} \theta_j + C_2 (|\Gamma| + |\PP|).
\]
\end{theorem}

\begin{corollary}[of Theorem \ref{thm:st-rn}]
\label{cor:st-plane-curves}
Let $\vec{q}$ be an algebraically independent Pfaffian chain of order $r$ and chain-degree $\alpha$. Let $\PP \subseteq \R^2$ be a set of points, and let $\Gamma \subseteq \R^2$ be a set of distinct irreducible Pfaffian curves such that each $\gamma \in \Gamma$ is defined by at most $m$ Pfaffian functions from $\Pf_n(\beta, \vec{q})$. Suppose that $G_{\PP, \Gamma}$ does not contain any copy of $K_{s,t}$. Then, for any $\eps > 0$, there exist constants $C_1 = C_1(\eps, m, \alpha, \beta, r, s, t)$ and $C_2 = C_2(\eps, m, \alpha, \beta, r, s, t)$ such that 
\[
|I(\PP, \Gamma)| \le C_1 \left(|\PP|^{\frac{s(r+1)}{s(r+2) - 1} + \eps}|\Gamma|^{\frac{(s-1)(r+2)}{s(r+2) - 1}}\right) + C_2 \left(|\PP| + |\Gamma|\right).
\]
\end{corollary}

\begin{proof}[Proof of Corollary \ref{cor:st-plane-curves}]
    Immediate by setting $n=2$ and $\theta_j = 0$ in Theorem \ref{thm:st-rn}
\end{proof}

\begin{remark}
    Theorem \ref{thm:st-rn} works under the assumption that low degree Pfaffian sets do not contain too many curves from $\Gamma$. This assumption matches the assumption of the main incidence theorems in \cite{sharirrd2016incidencessolomon} (henceforth this will be called the \emph{SSS assumption}), where they study incidences between points and algebraic curves in $\R^{n}$ when $n \ge 3$. 
    
    In the algebraic case, when the ambient space is $\R^2$, an SSS type assumption can be avoided simply because it is well known that an algebraic curve in $\R^2$ of degree $D$ contains at most $D$ irreducible components. However, in the Pfaffian case, the question of counting the number irreducible components becomes complicated. Counting irreducible components has been achieved only in restricted cases; for e.g. \cite{roy1994finding} study irreducible (over the real algebraic numbers) components of exponential polynomials. However, using Corollary \ref{cor:irreducible-components} where we obtain a coarse bound on the number of irreducible components, we are able to avoid the SSS assumption in the statement of Corollary \ref{cor:st-plane-curves}.
\end{remark}

\subsection{Pfaffian Joints}
\label{sec:applications-theoremstatements-joints}

\begin{definition}[Joints] \label{defn:joints}
Let $m, n \ge 2$. For $i \in [m]$, suppose $\Gamma_i$ is a set of $k_i$-dimensional semi-Pfaffian sets in $\R^n$ such that $\sum_{i=1}^m k_i = n$. A point $x \in \R^n$ is a joint if
\begin{enumerate}
\item for each $i \in [m]$, there exists $\gamma_i \in \Gamma_i$ such that $x$ is a smooth point of $\gamma_i$, and
\item the tangent spaces of $\gamma_i$ at $x$ for all $i$ span $\R^n$.
\end{enumerate}
\end{definition}

The following theorem establishes an upper bound on the number of joints formed by $n$ sets of Pfaffian curves in $\R^n$.

\begin{theorem}
\label{thm:pfaffian-joints}
Let $\vec{q}$ be an algebraically independent Pfaffian chain of order $r$ and chain-degree $\alpha$. Suppose there are $n \ge 3$ sets of distinct irreducible Pfaffian curves $\{\Gamma_i\}_{i \in [n]}$ in $\R^n$ such that for all $i \in [n]$, each $\Pfaff_i \in \Gamma_i$ is defined by at most $m$ Pfaffian functions from $\Pf_n(\beta, \vec{q})$. Then for any $\eps > 0$, there exists a constant $C = C(\eps, n, m, \alpha, \beta, r)$ such that the number of joints is bounded by 
\[
C\cdot \min_{j \in [n]} \left\{
|\Gamma_j|^{\eps}\right\} \cdot \prod_{i = 1}^n |\Gamma_i|^{\max\left\{\frac{n + r}{n(n - 1)}, \frac{2}{n+1}\right\}}.
\]
\end{theorem}

\begin{remark} 
Using Lemma \ref{lem:part}, one can remove the requirement of irreducibility of the Pfaffian curves in the hypothesis of Theorem \ref{thm:st-rn}, and Corollary \ref{cor:st-plane-curves}, and make it work for arbitrary semi-Pfaffian curves. The same holds for Theorem \ref{thm:pfaffian-joints} as well.
\end{remark}

\subsection{Discussion}
\label{sec:applications-discussion}

\subsubsection{Non-degeneracy assumptions}

In our Szemer\'edi-Trotter type Theorems, we assume that the incidence graph does not contain any copy of $K_{s,t}$. Suppose $\Gamma$ in Theorem \ref{thm:st-rn} was just a set of lines, then it is true that $K_{2,2} \not\subseteq G_{\PP, \Gamma}$. Thus Theorem \ref{thm:st-rn} generalizes the classical version of the Szemer\'edi-Trotter theorem \cite{szemereditrotter}. Similarly, if $\Gamma$ is purely a set of unit circles, then $G_{\PP, \Gamma}$ contains no copy of $K_{2,3}$, or if $\Gamma$ is a set of arbitrary circles, $G_{\PP, \Gamma}$ contains no copy of $K_{3,2}$. 

In general, if $\Gamma$ is a general set of varieties, it is quite natural to require some kind of non-degeneracy assumptions on $\PP$ and $\Gamma$ to prove non-trivial incidence bounds. For e.g., in \cite{solymosi2012incidence}, a bound on incidences between points and higher dimensional varieties is proved assuming some axioms that they call \emph{pseudoline axioms} -- i.e. axioms which ensure that $\Gamma$ behaves somewhat like a set of lines. \cite[Theorem 1.1]{pachsharir1998pointscurves} proves an incidence bound for points and simple curves in $\R^2$ with an assumption equivalent to not having $K_{s,t}$ inside $G_{\PP, \Gamma}$, although it is worded slightly differently. \cite[Theorem 1.4 and Theorem 1.5]{sharirrd2016incidencessolomon} prove incidence bounds for points and curves in $\R^d$, $d \ge 3$, by making an assumption that ensures that no partitioning hypersurface can contain too many elements of $\Gamma$; we include the same assumption in Theorem \ref{thm:st-rn}. \cite{basu2018minimal} prove an o-minimal \sztr{} theorem for points and plane curves assuming there is no copy of $K_{2, t}$ for some fixed $t$.

An example of a degenerate case is where $\Gamma$ is a set of planes all intersecting in the same line (one might visualise a partially open book where the pages are the planes, and the spine of the book is the common line of intersection), and all $\PP$ lie on the line. Then the number of incidences is exactly $|P||\Gamma|$, which is the trivial upper bound.

\subsubsection{Degree of partitioning polynomial or partitioning Pfaffian function} 
\label{sec:constantsvshighdegree}

When applying polynomial partitioning, or in our case, perhaps Pfaffian partitioning, one needs to decide what the degree of the polynomial or Pfaffian function must be. In incidence geometry literature, there is a dichotomy in this choice - high degree partitioning vs low/constant degree partitioning. 

For instance, the sharp version of the original \sztr{} theorem, where we count incidences between points and lines in $\R^2$, can be proved using polynomial partitioning where the degree of the partitioning polynomial $P$ is high, i.e. the degree depends on the cardinalities of $\PP$ and $\Gamma$ \cite{kaplan2012simple}. Since the degree of $P$ is high, the number of points/lines in each connected component of $\R^2 \setminus Z(P)$ is small; this allows us to use the naive bound of Theorem \ref{thm:incidence-naive} to count the number of incidences in each connected component, and subsequently amplify it. However, the higher degree of $P$ in turn means that the number of incidences on $Z(P)$ becomes higher. While there is no trouble in $\R^2$, using the same strategy in higher dimensions causes trouble.

\cite{solymosi2012incidence} introduced the idea of partitioning with a constant degree polynomial. But for a loss in power of $\eps$, this strategy allowed them to prove sharp bounds on incidences in higher dimensions between points and high-dimensional varieties. In this case, while counting incidences on the zero set of the partitioning polynomial becomes easier, one can no longer use the naive bound of Theorem \ref{thm:incidence-naive} to count incidences in connected components; one has to proceed by induction with a strong enough hypothesis. In fact, we lose a factor $\eps$ in the exponent precisely to close the induction.

In Theorem \ref{thm:st-rn} and Theorem \ref{thm:pfaffian-joints}, we use constant degree partitioning, and particularly, we use constant degree Pfaffian partitioning. While counting incidences on Pfaffian sets is indeed harder than on algebraic sets, the fact that the degree of the partitioning Pfaffian function is constant assuages this problem.

\subsubsection{Projections of Pfaffian sets}
\label{sec:projection-pfaffian}

While one might hope to generalize Theorem \ref{thm:st-rn} to incidences between points and higher dimensional semi-Pfaffian sets, and generalize Theorem \ref{thm:pfaffian-joints} to counting joints between higher dimensional Pfaffian sets, there is one common issue that stymies both directions, i.e. the issue of projections. While semi-algebraic sets project into semi-algebraic sets by the Tarski-Seidenberg theorem, projections of semi-Pfaffian sets, called sub-Pfaffian sets, need not be semi-Pfaffian. Such an example is provided in \cite{osgood}.

For an example of how this issue manifests, suppose we wished to count incidences between points and two-dimensional semi-Pfaffian sets in $\R^3$. In this case, counting incidences on the zero set of the partitioning Pfaffian function becomes difficult; we cannot induct using Corollary \ref{cor:st-plane-curves} because two-dimensional semi-Pfaffian sets which do not lie on the partitioning Pfaffian set intersect in one-dimensional Pfaffian sets, and we cannot project these one-dimensional intersections down to $\R^2$ because of the problem mentioned above.

One way around this problem could be to develop multi-level polynomial/Pfaffian partitioning theorems, which is a challenging problem with many works devoted to it \cites{matousekmultilevel2015,basusombrapolynomial2016,zaranfpsz2017,walshpolynomial2020}. However, this is out of the scope of our work, and we leave it for future work.

\subsection{Proofs of Theorem \ref{thm:st-rn} and Theorem \ref{thm:pfaffian-joints}}
\label{sec:applications-proofs}

\subsubsection{Proof of \sztr{} type theorem}
To prove our \sztr{} type theorems, we make use of the following naive bound on incidences between points and an arbitrary set of sets in $\R^n$. In fact, this works not just over $\R$, but any field.

\begin{theorem}[{K\H ov\'ari-S\'os-Tur\'an Theorem \cite[Theorem 4.5.2]{matousekdiscrete}}] \label{thm:incidence-naive}
Let $\PP$ be a set of points in $\R^n$ and let $\Gamma$ be a set of subsets of $\R^n$. Suppose that $G_{\PP, \Gamma}$ does not contain any copy of $K_{s,t}$. Then there exists a constant $C = C(s, t)$ such that \[|I(\PP, \Gamma)| \le C \cdot \min\left(|\PP||\Gamma|^{\frac{s-1}{s}} + |\Gamma|, |\PP|^{\frac{t-1}{t}}|\Gamma| + |\PP|\right).\]
\end{theorem}

\begin{proof}[Proof of Theorem \ref{thm:st-rn}]
In general, various intermediary constants that depend on $\eps$, $n$, $m$, $\alpha$, $\beta$, $r$, $s$, and $t$, shall all be denoted by $\kappa$ without disambiguation; for e.g., $\kappa^2$ in some step will be replaced with $\kappa$ in the next step. Only when we need to disambiguate, we will use $\kappa_i$. Later in the proof, we will use Theorem \ref{thm:pfaff-part-bbz} to partition $\PP$ and $\Gamma$ simultaneously; to distinguish, we let $C_0$ denote the maximum of all constants in Theorem \ref{thm:pfaff-part-bbz}. Also, let $C'$ denote the constant of Theorem \ref{thm:incidence-naive}.

We will proceed by induction on $|\PP| + |\Gamma|$. The base case of $|\PP| + |\Gamma| = \kappa$ can be handled by using Theorem \ref{thm:incidence-naive}, i.e.,
\[
|I(\PP, \Gamma)| \le C'(\kappa^{\frac{s-1}{s}}\kappa + \kappa) \le \kappa_1,
\]
thus setting $C_2 \ge \kappa_1$ proves the base case.

For the induction step, we use Theorem \ref{thm:pfaff-part-bbz} to infer the existence of $P \in \Pf_n(D, \vec{q})$, with $D = D(\eps, n, m, \alpha, \beta, r, s, t)$ to be specified later, such that $\R^n \setminus Z(P)$ has connected components $\conncomp_1, \ldots, \conncomp_{\alpha}$ satisfying, for all $i \in [\alpha]$,
\[
|\PP \cap \conncomp_i| \le C_0\frac{|\PP|}{D^{n+r}}, \qquad \text{and} \qquad \left|\{\gamma \in \Gamma \suchthat \gamma \cap \conncomp_i \neq \emptyset\}\right| \le C_0\frac{|\Gamma|}{D^{n-1}}.
\]
 
Setting $D \ge C_0^{\frac{1}{n+r}} + 1$, we have that $C_0\frac{|\PP|}{D^{n+r}} < |\PP|$; thus we can use the induction hypothesis to obtain that the number of incidences in $\conncomp_i$ is bounded as
\begin{align} 
&|I(\PP, \Gamma) \cap \conncomp_i| \nonumber \\
&\le C_1 \left(\left(\frac{C_0|\PP|}{D^{n+r}}\right)^{\frac{s(r+1)}{s(n+r)-n+1} + \eps}\left(\frac{C_0|\Gamma|}{D^{n-1}}\right)^{\frac{(s-1)(n+r)}{s(n+r) -n + 1}} \right) + |\PP \cap \conncomp_i|\sum_{j=2}^{n-1} \theta_j \nonumber \\
& \qquad \qquad \qquad + C_2\left(|\PP \cap \mathcal{C}_i| + \frac{C_0|\Gamma|}{D^{n-1}}\right) \nonumber \\
&= C_1 C_0^{\frac{s(r+1) + (s-1)(n+r)}{s(n+r) - n + 1} + \eps} \left(\frac{|\PP|}{D^{n+r}}\right)^{\frac{s(r+1)}{s(n+r) - n + 1} + \eps}\left(\frac{|\Gamma|}{D^{n-1}}\right)^{\frac{(s-1)(n+r)}{s(n+r)-n + 1}}  + |\PP \cap \conncomp_i|\sum_{j=2}^{n-1} \theta_j  \nonumber \\
& \qquad \qquad \qquad+ C_2\left(|\PP \cap \mathcal{C}_i| + \frac{C_0|\Gamma|}{D^{n-1}}\right) \nonumber \\
&= \kappa C_1 \left(\frac{|\PP|^{\frac{s(r+1)}{s(n+r) - n + 1} + \eps}|\Gamma|^{\frac{(s-1)(n+r)}{s(n+r)-n + 1}}}{D^{\frac{s(n+r)(r+1) + (s-1)(n+r)(n-1)}{s(n+r) - n + 1} + (n+r)\eps}}\right) + |\PP \cap \conncomp_i|\sum_{j=2}^{n-1} \theta_j \nonumber \\
&\qquad \qquad \qquad  + C_2\left(|\PP \cap \mathcal{C}_i| + \frac{C_0|\Gamma|}{D^{n-1}}\right) \nonumber \\
&= \kappa C_1 D^{-(n+r)(1 + \eps)}|\PP|^{\frac{s(r+1)}{s(n+r) - n+1} + \eps}|\Gamma|^{\frac{(s-1)(n+r)}{s(n+r)-n+1}} + |\PP \cap \conncomp_i|\sum_{j=2}^{n-1} \theta_j \nonumber \\
&\qquad \qquad \qquad + C_2\left(|\PP \cap \mathcal{C}_i| + \frac{C_0|\Gamma|}{D^{n-1}}\right). \label{eqn:Icc-rn}
\end{align}

If $|P| < |\Gamma|^{\frac{1}{s}}$, then from Theorem \ref{thm:incidence-naive} we'd have that
\[
|I(\PP, \Gamma)| \le C'(|\PP||\Gamma|^{\frac{s-1}{s}} + |\Gamma|) < 2C'|\Gamma|,
\]
and as long as $C_2 \ge 2C'$, we would be done. Thus without loss of generality, we can assume
\begin{equation}
\label{eqn:cases-to-consider-wlog-rn}
|\Gamma|^{\frac{1}{s}} \le |\PP|.
\end{equation}
Now, assuming \eqref{eqn:cases-to-consider-wlog-rn}, we have
\begin{equation} 
\label{eqn:gamma-ub-wlog-rn}
|\Gamma| = |\Gamma|^{\frac{(s-1)(n+r)}{s(n+r) - n + 1}}|\Gamma|^{\frac{(r+1)}{s(n+r) - n+1}} \le |\Gamma|^{\frac{(s-1)(n+r)}{s(n+r) - n+1}}|\PP|^{\frac{s(r+1)}{s(n+r) - n + 1}}.
\end{equation}

Noting that by Theorem \ref{thm:bezout-pfaffian} (and also \cite[Theorem 3.4]{gabrielov2004complexity}), we have that $\R^n \setminus Z(P)$ has $\alpha \le \kappa D^{n+r}$ connected components, we estimate the total number of incidences in all the connected components of $\R^n \setminus Z(P)$ as follows:
\begin{align} 
&\left|I(\PP, \Gamma) \cap \left(\bigcup_{i=1}^{\alpha} \conncomp_i\right)\right| \nonumber \\
&= \sum_{i=1}^{\alpha} \left|I(\PP, \Gamma) \cap \conncomp_i\right| \nonumber \\
&\le \sum_{i=1}^{\alpha} \left(\kappa C_1 D^{-(n+r)(1 + \eps)}|\PP|^{\frac{s(r+1)}{s(n+r) - n+1} + \eps}|\Gamma|^{\frac{(s-1)(n+r)}{s(n+r)-n+1}}\right) \nonumber \\
&\qquad \qquad \qquad + \sum_{i=1}^{\alpha} \left(|\PP \cap \conncomp_i|\sum_{j=2}^{n-1} \theta_j + C_2\left(|\PP \cap \mathcal{C}_i| + \frac{C_0|\Gamma|}{D^{n-1}}\right)\right) \eqcomment{by \eqref{eqn:Icc-rn}} \nonumber \\
&= \alpha \cdot \kappa C_1 D^{-(n+r)(1 + \eps)}|\PP|^{\frac{s(r+1)}{s(n+r) - n+1} + \eps}|\Gamma|^{\frac{(s-1)(n+r)}{s(n+r)-n+1}}  \nonumber \\
&\qquad \qquad \qquad + \sum_{i=1}^{\alpha} \left(|\PP \cap \conncomp_i|\sum_{j=2}^{n-1} \theta_j + C_2\left(|\PP \cap \mathcal{C}_i| + \frac{C_0|\Gamma|}{D^{n-1}}\right)\right) \nonumber \\
&\le \kappa C_1 D^{-(n+r)\eps}|\PP|^{\frac{s(r+1)}{s(n+r) - n+1} + \eps}|\Gamma|^{\frac{(s-1)(n+r)}{s(n+r)-n+1}} + \left|\PP \cap \left(\bigcup _{i=1}^{\alpha} \conncomp_i\right)\right|\sum_{j=2}^{n-1} \theta_j \nonumber \\
&\qquad \qquad \qquad + C_2\left|\PP \cap \left(\bigcup_{i=1}^{\alpha} \conncomp_i \right)\right| + \alpha C_2\left(\frac{C_0|\Gamma|}{D^{n-1}}\right) \nonumber \\
&\le \kappa_2 C_1 D^{-(n+r)\eps}|\PP|^{\frac{s(r+1)}{s(n+r) - n+1} + \eps}|\Gamma|^{\frac{(s-1)(n+r)}{s(n+r)-n+1}} + \left|\PP \cap \left(\bigcup _{i=1}^{\alpha} \conncomp_i\right)\right|\sum_{j=2}^{n-1} \theta_j \nonumber \\
&\qquad \qquad \qquad  + C_2\left|\PP \cap \left(\bigcup_{i=1}^{\alpha} \conncomp_i \right)\right| + \kappa D^{r+1}C_2|\Gamma| \nonumber \\
&\le \kappa_2 C_1 D^{-(n+r)\eps}|\PP|^{\frac{s(r+1)}{s(n+r) - n+1} + \eps}|\Gamma|^{\frac{(s-1)(n+r)}{s(n+r)-n+1}} + \left|\PP \cap \left(\bigcup _{i=1}^{\alpha} \conncomp_i\right)\right|\sum_{j=2}^{n-1} \theta_j  \nonumber \\
&\qquad \qquad \qquad + C_2\left|\PP \cap \left(\bigcup_{i=1}^{\alpha} \conncomp_i \right)\right| + \kappa_3 D^{r+1}C_2|\Gamma|^{\frac{(s-1)(n+r)}{s(n+r) - n+1}}|\PP|^{\frac{s(r+1)}{s(n+r) - n + 1}} \eqcomment{by \eqref{eqn:gamma-ub-wlog-rn}} \nonumber \\
&\le \frac{C_1}{4}|\PP|^{\frac{s(r+1)}{s(n+r) - n+1} + \eps}|\Gamma|^{\frac{(s-1)(n+r)}{s(n+r)-n+1}} + \left|\PP \cap \left(\bigcup _{i=1}^{\alpha} \conncomp_i\right)\right|\sum_{j=2}^{n-1} \theta_j + C_2 \left|\PP \cap \left(\bigcup _{i=1}^{\alpha} \conncomp_i\right)\right| , \label{eqn:incidences-in-cc-rn}
\end{align}
where the last line holds as long $D \ge (8\kappa_2)^{\frac{1}{(n+r)\eps}}$ and $C_1 \ge 8\kappa_3D^{r+1}C_2$.

Define $\Gamma' := \{\gamma \in \Gamma \suchthat \gamma \subseteq Z(P)\}$. Apply Corollary \ref{cor:smooth-decomposition} to decompose $Z(P)$, and let $Z_1$ denote the union of all the irreducible $1$-dimensional components in the decomposition. Define $\Gamma_1' := \{\gamma \in \Gamma \suchthat \gamma \subseteq Z_1\} \subseteq \Gamma'$. Now, by Corollary \ref{cor:irreducible-components}, we have that the number of irreducible 1-dimensional components of $Z_1$ is at most $\kappa D^{\kappa}$, thus we have that
\begin{equation} \label{eqn:incidences-on-zp-1-from-inside-rn}
|I(\PP \cap Z(P), \Gamma_1')| \le |\PP \cap Z(P)| \cdot |\Gamma_1'| \le \kappa |\PP \cap Z(P)| D^{\kappa} = \kappa_4|\PP \cap Z(P)|.
\end{equation}

For all $1 < j < n$, let $Z_j$ denote the union of all the irreducible $j$-dimensional components of $Z(P)$. Correspondingly, define $\Gamma_j' := \{\gamma \in \Gamma \suchthat \gamma \subseteq Z_j\} \subseteq \Gamma'$. By assumption, we have
\begin{equation} \label{eqn:incidences-on-zp-j-from-inside-rn}
|I(\PP \cap Z(P), \Gamma_j')| \le |\PP \cap Z(P)| \cdot |\Gamma_j'| \le \theta_j |\PP \cap Z(P)|.
\end{equation}

Also, by Lemma \ref{lem:part} and Corollary \ref{cor:khovanskii}, we have that each $\gamma \in \Gamma \setminus \Gamma'$ intersects $Z(P)$ in at most $\kappa D^{r+1}$ points, thus
\begin{equation}
    \label{eqn:incidences-on-zp-from-outside-rn}
    |I(\PP \cap Z(P), \Gamma \setminus \Gamma')| \le \kappa|\Gamma \setminus \Gamma'|D^{r+1} \le \kappa_5 |\Gamma|. 
\end{equation}

Finally, from \eqref{eqn:incidences-in-cc-rn}, \eqref{eqn:incidences-on-zp-1-from-inside-rn}, \eqref{eqn:incidences-on-zp-from-outside-rn},and \eqref{eqn:incidences-on-zp-from-outside-rn}, we deduce
\begin{align}
    &|I(\PP, \Gamma)| \nonumber \\
    &= \left|I(\PP, \Gamma) \cap \left(\bigcup _{i=1}^{\alpha} \conncomp_i\right)\right| + \sum_{j=1}^{n-1} |I(\PP \cap Z(P), \Gamma_j')| + |I(\PP \cap Z(P), \Gamma \setminus \Gamma')| \nonumber \\
    &\le \frac{C_1}{4}|\PP|^{\frac{s(r+1)}{s(n+r) - n+1} + \eps}|\Gamma|^{\frac{(s-1)(n+r)}{s(n+r)-n+1}} + \left|\PP \cap \left(\bigcup _{i=1}^{\alpha} \conncomp_i\right)\right|\sum_{j=2}^{n-1} \theta_j + C_2 \left|\PP \cap \left(\bigcup _{i=1}^{\alpha} \conncomp_i\right)\right| \nonumber \\
    &\qquad \qquad \qquad + |I(\PP \cap Z(P), \Gamma_1')| + \sum_{j=2}^{n-1} |I(\PP \cap Z(P), \Gamma_j')| + |I(\PP \cap Z(P), \Gamma \setminus \Gamma')| \nonumber \\
    &\le \frac{C_1}{4}|\PP|^{\frac{s(r+1)}{s(n+r) - n+1} + \eps}|\Gamma|^{\frac{(s-1)(n+r)}{s(n+r)-n+1}} + \left|\PP \cap \left(\bigcup _{i=1}^{\alpha} \conncomp_i\right)\right|\sum_{j=2}^{n-1} \theta_j + C_2 \left|\PP \cap \left(\bigcup _{i=1}^{\alpha} \conncomp_i\right)\right| \nonumber \\
    &\qquad \qquad \qquad + \kappa_4|\PP \cap Z(P)| + \sum_{j=2}^{n-1} \theta_j |\PP \cap Z(P)| + \kappa_5 |\Gamma| \nonumber \\
    &\le \frac{C_1}{4}|\PP|^{\frac{s(r+1)}{s(n+r) - n+1} + \eps}|\Gamma|^{\frac{(s-1)(n+r)}{s(n+r)-n+1}} + \left|\PP\right|\sum_{j=2}^{n-1} \theta_j + C_2 \left(\left|\PP\right| + |\Gamma| \right)\nonumber
\end{align}
where the last line holds as long as $C_2 \ge \max\{\kappa_4, \kappa_5\}$, thus closing the induction.
\end{proof}

\subsubsection{Proof of Pfaffian joints theorem}

\begin{proof}[Proof of Theorem \ref{thm:pfaffian-joints}]
As before, we will use Theorem \ref{thm:pfaff-part-bbz} to partition $\{\Gamma_i\}_{i \in [n]}$; to distinguish, we let $C_0$ denote the maximum of all constants in Theorem \ref{thm:pfaff-part-bbz}, and the constant we want to find by $C$. In general, various intermediary constants that depend on $\eps$, $n$, $m$, $\alpha$, $\beta$, and $r$, shall all be denoted by $\kappa$ without disambiguation; we will use constants $\kappa_i$ only when we need to disambiguate.

We use induction on the total number of elements in $\Gamma_1$. For the base case, let $|\Gamma_1| = \kappa$. Consider a particular curve $c_1 \in \Gamma_1$. With each curve $c_2 \in \Gamma_2$, by Theorem \ref{thm:khovanskii}$, c_1$ can have at most $\kappa$ intersections. This means that $c_1$ can have at most $\kappa|\Gamma_2|$ intersections with curves in $\Gamma_2$. Similarly, $c_1$ can have at most $\kappa|\Gamma_{i}|$ intersections with curves in $\Gamma_i$ for any $2 \le i \le n$. For a point on $c_1$ to be a joint, the point has to lie on a curve in each $\Gamma_i$ for $2 \le i \le n$. This in turn means that the number of joints on $c_1$ is bounded by \[\kappa\min\{|\Gamma_i|\}_{2 \le i \le n} \le \kappa \left(\prod_{i=2}^n |\Gamma_i|\right)^{\frac{1}{n-1}}.\] Since $|\Gamma_1| = \kappa$, we have that the number of joints formed by $\{\Gamma_i\}_{i \in [n]}$ is at most $\kappa_1 \left(\prod_{i=2}^n |\Gamma_i|\right)^{\frac{1}{n-1}}$; since $\frac{n+r}{n(n-1)} \ge \frac{1}{n-1}$, and $\frac{2}{n+1} \ge \frac{1}{n-1}$ whenever $n \ge 3$, this proves the base case as long as we choose $C \ge \kappa_1$.

Now suppose that the claim is true for all smaller values of $|\Gamma_1|$. For the induction step, we use Theorem \ref{thm:pfaff-part-bbz} to obtain a $P \in \Pf_n(D, \vec{q})$, with $D = D(\eps, n, m, \alpha, \beta, r)$ to be specified later, such that each connected component of $\R^n \setminus Z(P)$ intersects at most $C_0|\Gamma_i|D^{-(n-1)}$ Pfaffian curves from each $\Gamma_i$. Suppose $D \ge C_0^{\frac{1}{n-1}} + 1$. After partitioning, we then have that each connected component has at most $C_0|\Gamma_1|D^{-(n-1)} < |\Gamma_1|$ curves from $\Gamma_1$, and this means that we can use the induction hypothesis. 

Assume without loss of generality that $|\Gamma_1| = \min\left\{|\Gamma_i|\right\}_{i \in [n]}$, and for brevity, use \[\mu := \max\left\{\frac{n + r}{n(n - 1)}, \frac{2}{n+1}\right\}.\] Using the induction hypothesis, we have that in each connected component of $\R^n \setminus Z(P)$, there are at most $\mathcal{J}_{cc}$ joints, where
\begin{align}
\mathcal{J}_{cc} &\le C \cdot \left(C_0|\Gamma_1|D^{-(n-1)}\right)^{\eps} \cdot \prod_{i = 1}^n \left(C_0|\Gamma_i|D^{-(n-1)}\right)^{\mu} \nonumber \\
&= C\cdot C_0^{\eps + n\mu}\cdot \left(|\Gamma_1|D^{-(n-1)}\right)^{\eps} \cdot \prod_{i = 1}^n \left(|\Gamma_i|D^{-(n-1)}\right)^{\mu}\nonumber \\
&= C\cdot C_0^{\eps + n\mu} \cdot \left(D^{-\eps(n-1) - n(n-1)\mu}\right)\cdot |\Gamma_1|^{\eps} \cdot \prod_{i = 1}^n |\Gamma_i|^{\mu} \nonumber \\
&\le C\cdot C_0^{\eps + n\mu} \cdot \left(D^{-\eps(n-1) - (n+r)}\right)\cdot |\Gamma_1|^{\eps} \cdot \prod_{i = 1}^n |\Gamma_i|^{\mu} \eqcomment{since $\mu \ge \frac{n+r}{n(n-1)}$}. \label{eqn:jcc}
\end{align}
By Theorem \ref{thm:bezout-pfaffian} (and also \cite[Theorem 3.4]{gabrielov2004complexity}), we have that $\R^n \setminus Z(P)$ has at most $\kappa D^{n + r}$ connected components, thus the sum total number of joints in all the connected components of $\R^n \setminus Z(P)$, say $\mathcal{J}_{cc}^{tot}$, is at most
\begin{align}
    \mathcal{J}_{cc}^{tot} & \le \kappa D^{n+r} \cdot \mathcal{J}_{cc} \nonumber \\
&\le \kappa D^{n+r} \cdot C\cdot C_0^{\eps + n\mu} \cdot \left(D^{-\eps(n-1) - (n+r)}\right)\cdot |\Gamma_1|^{\eps} \cdot \prod_{i = 1}^n |\Gamma_i|^{\mu} \eqcomment{by \eqref{eqn:jcc}} \nonumber \\
&= \kappa_2 \cdot C\cdot C_0^{\eps + n\mu} \cdot D^{-\eps(n-1)}\cdot |\Gamma_1|^{\eps} \cdot \prod_{i = 1}^n |\Gamma_i|^{\mu} \nonumber \\
&\le \frac{C}{4} \cdot |\Gamma_1|^{\eps} \cdot \prod_{i = 1}^n |\Gamma_i|^{\mu}, \label{eqn:jcc-tot}
\end{align}
where the last inequality is true as long as $D \ge \left(4\kappa_2 \cdot C_0^{\eps + n\mu}\right)^{\frac{1}{\eps(n-1)}}$.

Having counted the number of joints in $\R^n \setminus Z(P)$, we shall now count the number of joints on $Z(P)$. Using Corollary \ref{cor:smooth-decomposition}, we decompose $Z(P)$ into $\kappa$ number of Pfaffian sets, $Z_1, \ldots, Z_{\kappa}$, such that every joint on $Z(P)$ is on the smooth part of some $Z_i$. We will fix a $Z_a$ and count the number of joints on the smooth part of $Z_a$; the total number of joints on $Z(P)$ will just be at most $\kappa$ times the number of joints on the smooth part of $Z_a$.

If a joint is a smooth point of $Z_a$, then it is not possible that all the curves that form the joint are contained in $Z_a$. This is because the curves were indeed all contained in $Z_a$, then the span of their tangent spaces would have been contained in the tangent space of $Z_a$ at the same point, which is at most $k$-dimensional (by the guarantee of Corollary \ref{cor:smooth-decomposition}), thus contradicting the definition of joint. By Corollary \ref{cor:khovanskii}, we have that a curve that is not contained in $Z_a$ can have at most $\kappa D^{\kappa}$ intersections (the $\kappa$ values could be different, but we don't disambiguate since it doesn't effect the final result) with $Z_a$. This implies that the number of joints for which the curve from $\Gamma_1$ is not contained in $Z_a$ is bounded by $\kappa |\Gamma_1|$ (recall that $D$ is a constant in $\eps, n, m, \alpha, \beta, r$). Also, at the same time, for any $2 \le i_1 < i_2 \le n$, using Corollary \ref{cor:khovanskii}, we also know that the number of joints is bounded by $\kappa |\Gamma_{i_1}||\Gamma_{i_2}|$; this is because each curve in $\Gamma_{i_1}$ can intersect each curve in $\Gamma_{i_2}$ in at most $\kappa$ points.

Thus we have that the number of joints on $\smooth{Z_a}$ for which the curve from $\Gamma_1$ is not contained in $Z_a$ is bounded by 
\begin{align}
\min\left\{\kappa |\Gamma_1|, \{\kappa|\Gamma_{i_1}||\Gamma_{i_2}|\}_{2 \le i_1 < i_2 \le n}\right\} &\le \kappa \cdot |\Gamma_1|^{\frac{2}{n+1}} \cdot \prod_{2 \le i_1 < i_2 \le n} \left(|\Gamma_{i_1}||\Gamma_{i_2}|\right)^{\frac{2}{(n+1)(n-2)}} \nonumber \\
&= \kappa \cdot |\Gamma_1|^{\frac{2}{n+1}}\cdot \prod_{i=2}^n \left(|\Gamma_i|^{(n-2)}\right)^{\frac{2}{(n+1)(n-2)}} \nonumber \\
&= \kappa \prod_{i=1}^n |\Gamma_i|^{\frac{2}{n+1}}, \nonumber
\end{align}
and consequently, the total number of joints on $\smooth{Z_a}$ is at most $\kappa \prod_{i=1}^n |\Gamma_i|^{\frac{2}{n+1}}$. Finally, since $Z(P)$ is composed of $Z_1, \ldots, Z_{\kappa}$, the total number of joints on $Z(P)$, say $\mathcal{J}_{Z(P)}$, is at most 
\[ \mathcal{J}_{Z(P)} \le \kappa_3 \cdot \prod_{i=1}^n |\Gamma_i|^{\frac{2}{n+1}} \le \frac{C}{4} \cdot \prod_{i=1}^n |\Gamma_i|^{\frac{2}{n+1}} \le \frac{C}{4} \cdot \prod_{i=1}^n |\Gamma_i|^{\mu},
\]
where the second inequality follows as long as $C \ge 4\kappa_3$.

Finally, we get that the total number of joints is at most
\[\mathcal{J}_{cc}^{tot} + \mathcal{J}_{Z(P)} \le \frac{C}{4} \cdot |\Gamma_1|^{\eps} \cdot \prod_{i=1}^n |\Gamma_i|^{\mu} + \frac{C}{4} \cdot \prod_{i=1}^n |\Gamma_i|^{\mu} \le C \cdot |\Gamma_1|^{\eps} \cdot \prod_{i=1}^n |\Gamma_i|^{\mu}.\]
Thus as long as we choose 
\[
D \ge \max\left\{C_0^{\frac{1}{n-1}} + 1, \left(4\kappa_2 \cdot C_0^{\eps + n\mu}\right)^{\frac{1}{\eps(n-1)}}\right\} \text{ and } C \ge \max\left\{\kappa_1, 4\kappa_3\right\},
\]
we can close the induction.
\end{proof}

\textbf{Conflicts of interest:} None.

\bibliography{main}

\appendix
\section{Proof of Theorem \ref{thm:poly-part-bbz}}
\label{appendix:poly-part-bbz}
\begin{proof}[Proof of Theorem \ref{thm:poly-part-bbz}]
Recall that $\R[X_1, \ldots, X_n]_{(\le d)}$ has dimension $\binom{d + n}{n} > \frac{d^n}{n!}$. For $\ell \in \N_{> 0}$, define $d_{\ell}:= \left\lceil \left( m\cdot n!\cdot 2^{\ell-1}\right)^{1/n}\right\rceil$, so that

\begin{equation} \label{eqn:delta-l}
m2^{\ell-1} \le \frac{d_{\ell}^n}{n!} < m2^{n + \ell - 1}.
\end{equation} 

Let $s$ be the smallest integer such that $\sum_{\ell=1}^s d_{\ell} \le D < \sum_{\ell=1}^{s+1} d_{\ell}$. Note that we do not exclude $s=0$. In particular,
\begin{equation} \label{eqn:D-delta-l}
D < \sum_{i=1}^{s+1} d_{\ell} < 2(n!)^{\nicefrac{1}{n}}m^{\nicefrac{1}{n}}\frac{2^{\nicefrac{(s+1)}{n}}}{2^{\nicefrac{1}{n}} - 1} \;\;\;\; \Rightarrow \;\;\;\; D^n < \frac{2^{n+1}n!}{\left(2^{\nicefrac{1}{n}} - 1\right)^n}m2^s \;\;\;\; \Rightarrow \;\;\;\; \frac{1}{2^s} < \alpha_1 \frac{m}{D^n},
\end{equation}
where $\alpha_1 = \alpha_1(n)$ is a constant.

By \eqref{eqn:delta-l}, we have 
\[
\mathrm{dim}\, \R[X_1, \ldots, X_n]_{(\le d_{\ell})} > \frac{d_{\ell}^n}{n!}  \ge m 2^{\ell-1}
\]
for every $\ell$. Let $V_\ell$ be an arbitrary subspace of $\R[X_1, \ldots, X_n]_{(\le d_{\ell})}$ of dimension exactly $m2^{\ell-1} + 1$, and let $S(V_\ell)$ denote the unit sphere in $V_\ell$. Define 
\begin{equation*}
    Y := \prod_{\ell=1}^s S(V_\ell) \cong \prod_{\ell=1}^s S^{m2^{\ell-1}}.
\end{equation*}

For any polynomial $P \in \R[X_1, \ldots, X_n]$, let $C_P^{+} \subseteq \R^n$ denote the set of points in $\R^n$ where $P$ is positive, and let $C_P^{-}$ be the set of points where $P$ is negative, so that $\R^n \setminus Z(P) = C_P^{+} \cup C_P^{-}$ and $C_P^{-} \cap C_P^{+} = \emptyset$. Let $(P_1, \ldots, P_s) \in Y$ be a tuple of polynomials and set $P=P_1\cdots P_s$. It is clear from \eqref{eqn:D-delta-l} that $\mathrm{deg} \, P \le \sum_{\ell=1}^s d_{\ell} \le D$. For $\sigma = (\sigma_1, \ldots, \sigma_s) \in \Z_2^{s}$, we define the sign condition
\begin{equation*}
\mathcal{O}_{\sigma}^{(P_1,\dots,P_s)} := \bigcap_{i=1}^s C_{P_i}^{\sigma_i}.
\end{equation*}
This implies $\R^n \setminus Z(P) = \bigcup_{\sigma \in \Z_2^s} \mathcal{O}_{\sigma}^{(P_1,\ldots, P_s)}$ and each $\mathcal{O}_{\sigma}^{(P_1, \ldots, P_s)}$ is a union of some connected components of $\R^n \setminus Z(P)$.

For any $\sigma \in \Z_2^s$ and semi-Pfaffian set $\Pfaff \subseteq \R^n$, define the indicator function $I_{\sigma, \Pfaff}: Y \rightarrow \R$,
\begin{equation*}
I_{\sigma, \Pfaff}(P_1, \ldots, P_s)=\begin{cases}
			1, & \text{if $\mathcal{O}_{\sigma}^{(P_1, \ldots, P_s)} \cap \Pfaff \neq \emptyset$},\\
            0, & \text{if $\mathcal{O}_{\sigma}^{(P_1, \ldots, P_s)} \cap \Pfaff = \emptyset$}
		 \end{cases}.
\end{equation*}
Also, let $U_s \subseteq \R^{2^s}$ be the codimension 1 subspace defined as 
\begin{equation}
U_s = \left\{(y_{\sigma})_{\sigma \in \Z_2^s} \in R^{2^s} \suchthat \sum_{\sigma \in \Z_2^s} y_{\sigma} = 0\right\}.
\end{equation}
Consider the following map $\Phi: Y \rightarrow \bigoplus_{i=1}^m U_s$, which takes 
\begin{equation}
\label{eqn:phi}
(P_1, \ldots, P_s) \mapsto \left(\left(\sum_{\Pfaff \in \Gamma_i} I_{\sigma, \Pfaff}(P_1, \ldots, P_s) - \frac{1}{2^s} \sum_{\beta \in \Z_2^s} \sum_{\Pfaff \in \Gamma_i} I_{\beta, \Pfaff}(P_1, \ldots, P_s)\right)_{\sigma \in \Z_2^s}\right)_{i \in [m]}.
\end{equation} 
The sum $\sum_{\Pfaff \in \Gamma_i} I_{\sigma, \Pfaff}(P_1, \ldots, P_s)$ counts the number of varieties in $\Gamma_i$ that intersect the sign condition $\mathcal{O}_{\sigma}^{(P_1, \ldots, P_s)}$. Now, assume that the map $\Phi$ has a zero and pick $(P_1, \ldots, P_s) \in \Phi^{-1}(0)$. This means that each sign condition in $\{\mathcal{O}_{\sigma}^{(P_1, \ldots, P_s)}\}_{\sigma \in \Z_2^s}$ is intersected by the same number of elements of the set $\Gamma_i$, for all $i \in [m]$. 

For all $i \in [m]$ with $k_i > 0$, by Theorem \ref{thm:bezout-pfaffian}, we have that each $\Pfaff \in \Gamma_i$ intersects at most $CD^{k_i + r_i}$ connected components of $\R^n \setminus Z(P)$, and since each sign condition is a disjoint union of connected components of $\R^n \setminus Z(P)$, we have that each $\Pfaff \in \Gamma_i$ intersects at most $CD^{k_i + r_i}$ sign conditions. Since we have that $(P_1, \ldots, P_s) \in \Phi^{-1}(0)$, we have that, for all $i \in [m]$, each sign condition intersects at most $\frac{1}{2^s}|\Gamma_i|CD^{k_i + r_i} < C'|\Gamma_i|mD^{k_i + r_i - n}$ semi-Pfaffian sets from $\Gamma_i$, where we have used \eqref{eqn:D-delta-l} to obtain the upper bound. This proves Part \ref{pfaffpp-collection:dim-at-least-1}.

On the other hand, for all $i \in [m]$ with $k_i = 0$, each $\Pfaff \in \Gamma_i$ intersects at most $C$ connected components of $\R^n \setminus Z(P)$. In this specific case, we get that each connected component of $\R^n \setminus Z(P)$ intersects at most $C'|\Gamma_i|mD^{-n}$ elements of $\Gamma_i$. This proves Part \ref{pfaffpp-collection:dim-0}.

It remains to prove that $\Phi$ has a zero. Towards this, since the indicator functions $I_{\sigma, \Pfaff}$ are not continuous, we have to consider continuous approximations of the indicators. In \cite{guth2015polynomial} the following Lemma is proved which allows us to approximate $I_{\sigma, \Pfaff}$ by sequences of continuous functions.

\begin{lemma}[Lemma 3.1 in \cite{guth2015polynomial}] \label{lem:indicator-continuous} For every $\eps > 0$, $\Pfaff \subseteq \R^n$, and $\sigma \in \Z_2^s$, there exist functions $I_{\sigma, \Pfaff}^{\eps}: Y \rightarrow \R$ with the following properties.
\begin{enumerate}
\item $I_{\sigma, \Pfaff}^{\eps}$ are continuous.
\item $0 \le I_{\sigma, \Pfaff}^{\eps} \le 1$.
\item If $\mathcal{O}_{\sigma}^{(P_1, \ldots, P_s)} \cap \Pfaff = \emptyset$, then $I_{\sigma, \Pfaff}^{\eps} = 0$.
\item If $\eps_i \rightarrow 0$, $(P_1^i, \ldots, P_s^i) \rightarrow (P_1, \ldots, P_s)$ in $Y$ and $\mathcal{O}_{\sigma}^{(P_1, \ldots, P_s)} \cap \Pfaff \neq \emptyset$, then \[\lim_{i \rightarrow \infty} I_{\sigma, \Pfaff}^{\eps_i}(P_1^i, \ldots, P_s^i) = 1.\] In other words, $I_{\sigma, \Pfaff}(P_1, \ldots, P_s) \le \lim \inf_{i \rightarrow \infty} I_{\sigma, \Pfaff}^{\eps_i}(P_1^i, \ldots, P_s^i)$.
\end{enumerate}
\end{lemma}

For some sequence $\eps_i \rightarrow 0$, let $\Phi_i: Y \rightarrow U_s^{\oplus j}$ be defined as \[
(P_1, \ldots, P_s) \mapsto \left(\left(\sum_{\Pfaff \in \Gamma_i} I^{\eps_i}_{\sigma, \Pfaff}(P_1, \ldots, P_s) - \frac{1}{2^s} \sum_{\beta \in \Z_2^s} \sum_{\Pfaff \in \Gamma_i} I^{\eps_i}_{\beta, \Pfaff}(P_1, \ldots, P_s)\right)_{\sigma \in \Z_2^s}\right)_{i \in [m]}.
\]

Consider the following action of $\Z_2^s = \langle\omega_1, \ldots, \omega_s\rangle$ on $Y$: for all $1 \le l \le s$, given $(P_1, \ldots, P_s) \in Y$, the generator $\omega_l$ acts as
\begin{equation}
\label{eqn:z2-sign-flip}
\omega_l \circ (P_1, \ldots, P_s) = (P_1, \ldots, -P_l, \ldots, P_s).
\end{equation}
Also, consider the following action of $\Z^s_2$ on $\R^{2^s}$: an element $(\beta_1, \ldots, \beta_s) \in \Z_2^s$ acts on $(y_{\alpha})_{\alpha \in \Z_2^s}$ by acting on the indices as
\begin{equation}
\label{eqn:z2-indices-permute}
(\beta_1, \ldots, \beta_s) \circ (\alpha_1, \ldots, \alpha_s) = (\beta_1 + \alpha_1, \ldots, \alpha_s + \beta_s),
\end{equation}
where the addition is in $\Z_2$. Thus the vector subspace $U_s$ is a $\Z_2^s$-subrepresentation of $\R^{2^s}$.

The functions $\Phi_i$ are continuous and are $\Z_2^s$-equivariant with respect to the above actions (\eqref{eqn:z2-sign-flip} and \eqref{eqn:z2-indices-permute}), assuming the diagonal action on the direct sum $U_s^{\oplus j}$. By the use of the $\Z_2^s$-equivariance of $\Phi_i$ and some tools from equivariant cohomology (Fadell-Husseini index theory \cite{fadellhusseini1988}), it is proved in \cite[Section 2.7]{blagojevic2017polynomial} that each $\Phi_i$ indeed has a zero.

Since each $\Phi_i$ has a zero, this means that for each $i$, there exists a $\rho_i \in Y$ such that $\Phi_i(\rho_i) = 0$. Because $Y$ is compact, this means that there is a converging subsequence \[
\lim_{i \rightarrow \infty} \rho_{\kappa(i)} = \rho \in Y,
\] 
where $\kappa: \N \rightarrow \N$ is a strictly increasing function. Since by Lemma \ref{lem:indicator-continuous} we have that the indicator functions $I_{\sigma, \Pfaff}$ are limits of sequences of continuous maps, we have that
\[
\Phi(\rho) = \lim_{i \rightarrow \infty} \Phi_{\kappa(i)}(\rho_{\kappa(i)}) = 0.
\]
This proves that $\Phi$ has a zero, and completes the proof.
\end{proof}

\end{document}